\documentclass{article}
\usepackage{mathrsfs}
\usepackage{graphicx}
\usepackage{bbm}
\usepackage{relsize}
\usepackage{enumerate}
\usepackage{amssymb,amsmath,amsthm}
\usepackage{bm}
\usepackage[title]{appendix}
\usepackage{url}

\usepackage{titlesec}
\usepackage{verbatim}
\usepackage{mathrsfs}
\usepackage[usenames, dvipsnames]{color}
\usepackage[letterpaper, portrait, margin=1in]{geometry}
\usepackage{amsmath, amssymb, amsthm}
\usepackage{graphicx}
\usepackage{mathrsfs}
\usepackage{xcolor}
\usepackage{bbm}
\usepackage{mathtools}
\usepackage[linesnumbered,lined,ruled]{algorithm2e}
\usepackage{algpseudocode}
\usepackage{chngcntr}
\usepackage{float,stmaryrd}

\theoremstyle{definition} 
\newtheorem{theorem}{Theorem}[section]

\newtheorem{lemma}[theorem]{Lemma}
\newtheorem{proposition}[theorem]{Proposition}

\newtheorem{example}[theorem]{Example}

\newtheorem{conjecture}[theorem]{Conjecture}
\newtheorem{question}[theorem]{Question}

\newcommand{\cart}{\boxempty}

\DeclareMathOperator{\val}{val}

\DeclareMathOperator{\gon}{gon}

\DeclarePairedDelimiter\abs{\lvert}{\rvert}

\DeclarePairedDelimiter\floor{\lfloor}{\rfloor}
\DeclarePairedDelimiter\norm{\lVert}{\rVert}
\DeclarePairedDelimiter\set{\{}{\}}
\DeclarePairedDelimiter\paren{(}{)}

\makeatletter
\let\oldabs\abs
\def\abs{\@ifstar{\oldabs}{\oldabs*}}
\let\oldnorm\norm
\def\norm{\@ifstar{\oldnorm}{\oldnorm*}}
\let\oldparen\paren
\def\paren{\@ifstar{\oldparen}{\oldparen*}}
\makeatother

\begin{document}


\title{On the gonality of Cartesian products of graphs}
\author{Ivan Aidun and Ralph Morrison}
\date{}


\maketitle

\begin{abstract}
In this paper we study Cartesian products of graphs and their divisorial gonality, which is a tropical version of the gonality of an algebraic curve.  We present an upper bound on the gonality of the Cartesian product of any two graphs, and provide instances where this bound holds with equality, including for the $m\times n$ rook's graph with $\min\{m,n\}\leq 5$.  We use our upper bound to prove that Baker's gonality conjecture holds for the Cartesian product of any two graphs with two or more vertices each, and we determine precisely which nontrivial product graphs have gonality equal to Baker's conjectural upper bound.
\end{abstract}

\section{Introduction}

In \cite{bn}, Baker and Norine introduced a theory of divisors on finite graphs in parallel to divisor theory on algebraic curves.  If $G=(V,E)$ is a connected multigraph, one treats $G$ as a discrete analog of an algebraic curve of genus $g(G)$, where $g(G)=|E|-|V|+1$.   This program was extended to metric graphs in \cite{gk} and \cite{mz}, and has been used to study algebraic curves through combinatorial means.

A divisor on a graph can be thought of as a configuration of poker chips on the vertices of the graph, where a negative number of chips indicates debt.  Equivalence of divisors is then defined in terms of chip-firing moves.  Each divisor $D$ has a \emph{degree}, which is the total number of chips; and a \emph{rank}, which measures how much added debt can be cancelled out by $D$ via chip-firing moves.



The \emph{gonality} of $G$ is the minimum degree of a rank $1$ divisor on $G$. This is one graph-theoretic analogue of the gonality of an algebraic curve \cite{cap}. In general, the gonality of a graph is NP-hard to compute \cite{gij}.  Nonetheless, we know the gonality of certain nice families of graphs:  the gonality of $G$ is $1$ if and only if $G$ is a tree \cite[Lemma 1.1]{bn2}; the complete graph $K_n$ has gonality $n-1$  for $n\geq 2$ \cite[Theorem 4.3]{db}; and the gonality of the complete $k$-partite graph $K_{n_1,\cdots n_k}$ is $\sum_{i=1}^k n_k-\max\{n_1,\cdots,n_k\}$ \cite[Example 3.2]{vddbg}.  One of the biggest open problems regarding the gonality of graphs is the following.

\begin{conjecture}[The gonality conjecture, \cite{baker}]\label{conjecture:gonality}  The gonality of a graph $G$ is at most $\floor{ \frac{g(G)+3}{2}}$.

\end{conjecture}
This conjecture has been confirmed for graphs with $g(G)\leq 5$ in \cite{ar}, with strong additional evidence coming from \cite{cd}.

In this paper, we study the gonality of the Cartesian product $G\cart H$ of two graphs $G$ and $H$.   This is the first such systematic treatment for these types of graphs, although many conjectures have been posed on the gonality of particular products \cite{treewidth, db, vddbg}. Our main result is that if $G$ and $H$ have at least two vertices each, then  $G\boxempty H$ satisfies Conjecture \ref{conjecture:gonality}.

\begin{theorem}\label{theorem:main}
Let $G$ and $H$ be connected graphs with at least two vertices each.  Then \[\textrm{gon}(G\boxempty H)\leq \left\lfloor\frac{g(G\boxempty H)+3}{2}\right\rfloor.\]
\end{theorem}

As a key step towards proving Theorem \ref{theorem:main}, we prove the following upper bound on the gonality of $G\boxempty H$.

\begin{proposition}\label{prop:upperbound}  For any two graphs $G$ and $H$,
\[\textrm{gon}(G\boxempty H)\leq\min\{\textrm{gon}(G)\cdot |V(H)|,\textrm{gon}(H)\cdot |V(G)|\}\]
\end{proposition}

For most examples of $G$ and $H$ where $\textrm{gon}(G\boxempty H)$ is known, the inequality is in fact an equality.  This leads us to pose the following question.

\begin{question}
\label{conjecture:product} For which graphs $G$ and $H$ do we have
\[\gon(G \boxempty H) = \min\set{\gon(G) \cdot \abs{V(H)},\gon(H) \cdot \abs{V(G)}}?\]
\end{question}

When a graph product $G \boxempty H$ has gonality $\min\set{\gon(G) \cdot \abs{V(H)},\gon(H) \cdot \abs{V(G)}}$, we say that it has the \emph{expected gonality}.  Some product graphs have gonality smaller than the expected gonality.  Let $G$ be a graph with three vertices $v_1,v_2$ and $v_3$, with edge multiset $\{v_1v_2,v_1v_2,v_2v_3\}$.  Since $g(G)=1$, we will see that $\gon(G)=2$ in Lemma \ref{lemma:g1}.  The expected gonality of $G\cart G$ is $\gon(G)\cdot |V(G)|=2\cdot 3=6$.  However, Figure \ref{figure:counterexample} illustrates three equivalent effective divisors of degree $5$ on $G\cart G$.  Since between the three divisors there is a chip on each vertex, any $-1$ debt can be eliminated wherever it is placed, so $G\cart G$ has a degree $5$ divisor of positive rank, and thus $\gon(G\cart G)\leq 5$.  In Proposition \ref{prop:arbitrarily_large} and Example \ref{example:simple}, we will see that the gap between gonality and expected gonality can in fact be arbitrarily large, both when considering simple and non-simple graphs.

\begin{figure}[hbt]
   		 \centering
\includegraphics[scale=1]{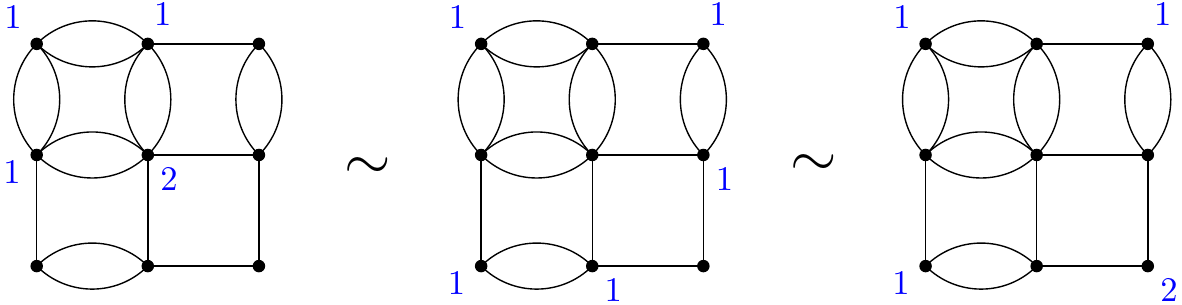}
	\caption{A positive rank divisor on $G\cart G$ with lower degree than expected}
	\label{figure:counterexample}
\end{figure}

Our paper is organized as follows.  In Section \ref{section:upperbound} we establish background and conventions and prove Proposition \ref{prop:upperbound}.  In Section \ref{section:evidence} we provide old and new instances where the equation in Question \ref{conjecture:product} is satisfied. In Section \ref{section:gonalityconjecture} we prove Theorem \ref{theorem:main}. We close in Section \ref{section:equality} by  determining when the gonality of a nontrivial product is equal to $\floor{(g+3)/2}$ in Theorem \ref{theorem:equality}.  It turns out that there are only finitely many such product graphs, $12$ simple and $11$ non-simple.

\section{Background and a proof of the upper bound}
\label{section:upperbound}

The main goal of this section is to prove our upper bound on $\textrm{gon}(G\boxempty H)$.  Before we do so we establish some definitions and notation.

Throughout this paper, a \emph{graph} is a connected multigraph, where we allow multiple edges between two vertices, but not edges from a vertex to itself.  We write $G=(V,E)$, where $V=V(G)$ is the set of vertices and $E=E(G)$ is the multiset of edges.  If every pair of vertices has at most one edge connecting them, we call $G$ \emph{simple}.
For any vertex $v\in V(G)$, the \emph{valence} of $v$, denoted $\textrm{val}(v)$, is the number of edges incident to $v$.  The \emph{genus} of $G$, denoted $g(G)$, is defined to be $|E|-|V|+1$.  Given two graphs $G=(V_1,E_1)$ and $H=(V_2,E_2)$, their \emph{Cartesian product} $G\boxempty H$ is the graph with vertex set $V_1\times V_2$, and $e$ edges connecting $(v_1,v_2)$ and $(w_1,w_2)$ if $v_1=w_1$ and $v_2$ is connected to $w_2$ in $H$ by $e$ edges, or if $v_2=w_2$ and $v_1$ is connected to $w_1$ in $G$ by $e$ edges.  A graph is called a \emph{non-trivial product} if it is of the form $G\boxempty H$, where $G$ and $H$ are graphs with at least two vertices each. The graph $G\boxempty H$ has $|V_1|\cdot|V_2|$ vertices and $|E_1|\cdot|V_2|+|E_2|\cdot|V_1|$ edges, so $g(G\boxempty H)=|E_1|\cdot|V_2|+|E_2|\cdot|V_1|-|V_1|\cdot|V_2|+1$.  An example of a product graph is illustrated in Figure \ref{figure:product_graph}.  This is the Cartesian product of the star tree $T$ with four vertices and the complete graph on $3$ vertices $K_3$.  There are three natural copies of $T$, one for each vertex of $K_3$; and there are four natural copies of $K_3$, one for each vertex of $T$.

\begin{figure}[hbt]
   		 \centering
\includegraphics[scale=1]{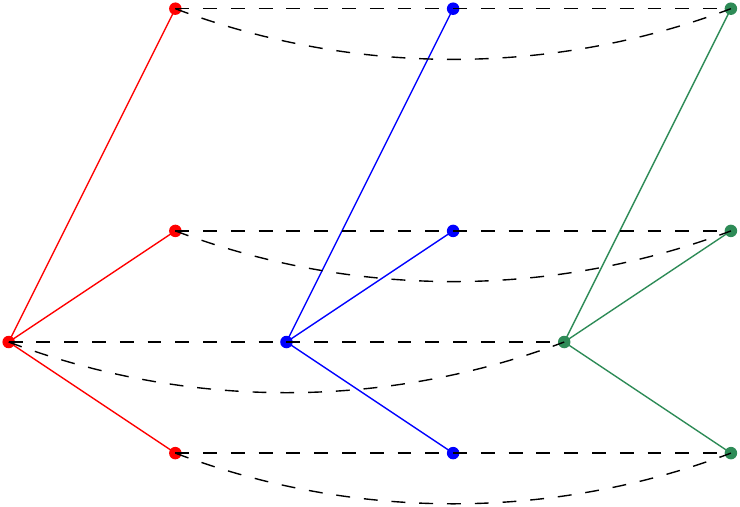}
	\caption{The Cartesian product of a tree with $K_3$}
	\label{figure:product_graph}
\end{figure}

A \emph{divisor} on a graph $G$ is a formal $\mathbb{Z}$-linear sum of the vertices of $G$:
$$\sum_{v\in V}a_v(v),\,\,\,\,a_v\in\mathbb{Z}.$$  The set of all divisors on a graph forms an abelian group, namely the free abelian group generated by the vertices of the graph. 
 The \emph{degree} of a divisor is the sum of the coefficients:
\[\deg\left(\sum_{v\in V}a_v(v)\right)=\sum_{v\in V}a_v.\]  In the language of chip configurations, the degree is the total number of chips present on the graph.  We say that a divisor is \emph{effective} if $a_v\geq 0$ for all $v\in V$, i.e. if no vertex is in debt. 

A \emph{chip-firing move} changes one divisor to another by \emph{firing} a vertex, causing it to donate chips to each neighboring vertex, one for each edge connecting the two vertices.  We say that two divisors are \emph{equivalent} to one another if they differ by a sequence of chip-firing moves, and write $D\sim D'$ if $D$ and $D'$ are equivalent divisors.

Let $D$ be a divisor on a graph $G$.  The \emph{rank} $r(D)$ of $D$ is the largest integer $r\geq 0$ such that, for all effective divisors $F$ of degree $r$, $D-F$ is equivalent to an effective divisor.  (If such an $r$ doesn't exist, we set $r(D)=-1$.) Note that if $D$ has non-negative rank, then it is equivalent to an effective divisor.  The theory of divisors on graphs mirrors the theory of divisors on algebraic curves, as illustrated in the following result.

\begin{theorem}
[The Riemann-Roch Theorem for graphs, \cite{bn}] \label{theorem:rr}

Let $D$ be a divisor on a graph $G$, and let $K$ be the divisor with $\textrm{val}(v)-2$ chips on each vertex $v$ of $G$.  Then
\[r(D)-r(K-D)=\deg(D)-g(G)+1.\]
\end{theorem}

The \emph{gonality} $\textrm{gon}(G)$ of a graph $G$ is the smallest degree of a divisor of positive rank.  Note that there always exists an effective divisor $D$ with $r(D)=\gon(G)$, since any divisor of non-negative rank is equivalent to an effective divisor. We can also define gonality in terms of a chip-firing game:  Player 1 places $k$ chips on the graph (for some $k$), and then Player 2 places $-1$ chips on the graph. If Player 1 can perform chip-firing moves to eliminate all debt from the graph, they win; otherwise, Player 2 wins.  The gonality of the graph is then the minimum $k$ such that Player 1 has a winning strategy.  For this reason, we refer to a divisor of positive rank as a \emph{winning divisor}.

As an application of the Riemann-Roch Theorem for graphs, we determine the gonality of any genus $1$ graph.

\begin{lemma}\label{lemma:g1}
If $g(G)=1$, then $\gon(G)=2$.
\end{lemma}

\begin{proof}
Let $G$ have genus $1$, and let $D$ be a divisor of degree $2$ on $G$.  Then
\[r(D)-r(K-D)=\deg(D)-g(G)+1=2-1+1=2.\]
Since $r(K-D)\geq -1$, we have $r(D)= 2+r(K-D)\geq 1$.  Thus $\gon(G)\leq 2$.  Since $G$ is not a tree, we have $\gon(G)>1$ by \cite[Lemma 1.1]{bn2}, so $\gon(G)=2$.
\end{proof}

To prove that the gonality of a graph is at most an integer $k$, it suffices to exhibit a divisor on $G$ of degree $k$ such that no matter where an opponent places a $-1$, debt may be eliminated from the graph via chip-firing.  It is this strategy we will use to prove Proposition \ref{prop:upperbound}.  (Providing a lower bound on the gonality of a graph is much more difficult, though some methods are available, as described in Section \ref{section:evidence}.)

\begin{proof}[Proof of Proposition \ref{prop:upperbound}]  We will show that there is a winning divisor $D$ on $G\boxempty H$ with $\deg(D)=\gon(G)\cdot\abs{V(H)}$, implying that $\gon(G \boxempty H)\leq \gon(G)\cdot\abs{V(H)}$.  By symmetry, we will also have $\gon(G \boxempty H)\leq \gon(H)\cdot\abs{V(G)}$

Let $F=\sum_{v\in V(G)} b_v(v)$ be a divisor on $G$ with $r(F)>0$ and $\deg(F)=\gon(G)$.  Let 
\[D=\sum_{(v,w)\in V(G)\times V(H)}a_{(v,w)}(v,w)\] be the divisor on $G\boxempty H$ defined by $a_{(v,w)}=b_v$ for all $v\in V(G)$ and $w\in V(H)$.  Note that the degree of $D$ is
\[\sum_{(v,w)\in V(G)\times V(H)} a_{(v,w)}=\sum_{w\in V(H)}\left(\sum_{v\in V(G)} b_v\right)=\sum_{w\in V(H)}\deg(F)=|V(H)|\cdot\deg(F).\]
In other words, $\deg(D)=\textrm{gon}(G)\cdot |V(H)|$.

To see that $D$ is a winning divisor, suppose the $-1$ chip is placed on the vertex $(v,w)$.  We may perform chip-firing moves on the copy $G \boxempty \{w\}$ as if we were playing on $G$ by doing the following: each time we would fire a vertex $v' \in G$, instead fire each vertex of the form $(v',u)$ where $u \in H$.  Since $F$ is a winning divisor on $G$, there is some sequence of chip-firing moves that removes all the debt from $G$, and so this substitution of chip-firing moves furnishes a sequence of chip-firing moves on $G \boxempty H$ that removes all the debt from $G \boxempty \{w\}$, and hence from all of $G \boxempty H$.
Thus, $D$ is a winning divisor, and $\gon(G \boxempty H)\leq \gon(G)\cdot\abs{V(H)}$.  By symmetry we conclude that
\[\gon(G \boxempty H) \leq \min\set{\gon(G)\cdot\abs{V(H)},\gon(H)\cdot\abs{V(G)}}.\]
\end{proof}

\begin{figure}[hbt]
   		 \centering
\includegraphics[scale=0.8]{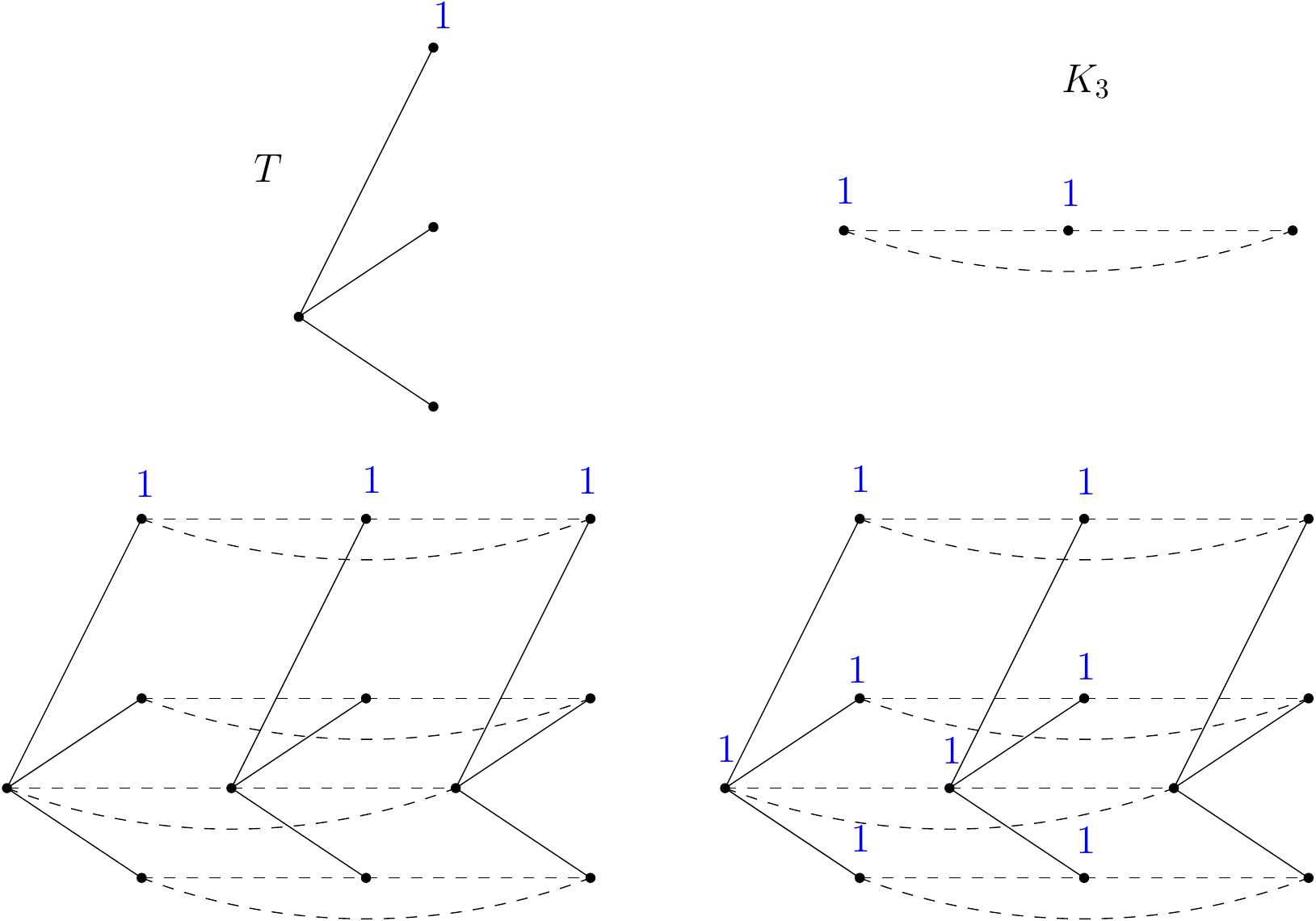}
	\caption{Positive rank divisors on $T$ and $K_3$, each of which yields a positive rank divisor on $T\boxempty K_3$}
	\label{figure:winning_divisor_on_products}
\end{figure}

The construction of $D$ from this proof is illustrated in Figure \ref{figure:winning_divisor_on_products} for the product of a tree $T$ with the complete graph $K_3$.  The top left illustrates a positive rank divisor on $T$ with degree equal to the gonality of $T$; we can build a positive rank divisor on $T\boxempty K_3$ by placing the same chips on each copy of $T$, as illustrated on the bottom left.  Since the number of copies of $T$ is equal to the number of vertices of $K_3$, the divisor on $T\boxempty K_3$ has degree $\textrm{gon}(T)\cdot |V(K_3)|=1\cdot 3=3$.  Similarly on the right we have a positive rank divisor on $K_3$ with degree equal to $\textrm{gon}(K_3)$, yielding a positive rank divisor on $T\boxempty K_3$ of degree $\textrm{gon}(K_3)\cdot |V(T)|=2\cdot 4=6$.  Both these divisors provide an upper bound on $\textrm{gon}(T\boxempty K_3)$, so $\textrm{gon}(T\boxempty K_3)\leq \min\{3,6\}=3$.  (It will follow from Proposition \ref{prop:tree_complete} that in fact $\textrm{gon}(T\boxempty K_3)=3$.)

Some graph products have gonality strictly smaller than the upper bound in Proposition \ref{prop:upperbound}. Indeed, the gap between gonality and expected gonality can be arbitrarily large, as shown in the following result.

\begin{proposition}\label{prop:arbitrarily_large}
For any $n\geq 2$, there exist a product graph $G\cart H$ with \[\min\{\textrm{gon}(G)\cdot |V(H)|,\textrm{gon}(H)\cdot |V(G)|\}-\gon(G\cart H)\geq n-1.\]
\end{proposition}

\begin{proof}
Given $n\geq 2$, construct a graph $G$ as follows.  Let $G$ have $n+1$ vertices $v_1,\cdots,v_n,v_{n+1}$, where $v_i$ and $v_{i+1}$ are connected by $n$ edges for $1\leq i\leq n-1$, and where $v_n$ and $v_{n+1}$ are connected by $1$ edge.  We claim that $\gon(G)=n$.  Certainly $(v_1)+\cdots+(v_{n-1})+(v_n)$ is a divisor of positive rank:  the only vertex on which $-1$ chips could be placed to introduce debt is $v_{n+1}$, and since $(v_n)\sim (v_{n+1})$ we have $(v_1)+\cdots+(v_{n-1})+(v_n)-(v_{n+1})\sim (v_1)+\cdots+(v_{n-1})$.  On the other hand, there exists no effective positive rank divisor of degree $n-1$, since with so few chips no chips could be moved between any two vertices, save for $v_{n}$ and $v_{n+1}$; and with $n-1$ chips, at least one of $v_1,\cdots,v_{n-1}$ and the pair $\{v_n,v_{n+1}\}$ would not have a chip, and so placing $-1$ chips there creates debt that cannot be eliminated.

Since $\gon(G)=n$ and $|V(G)|=n+1$, the expected gonality of $G\cart G$ is $n(n+1)$.  We now present a divisor of degree $n^2+1$, namely
\[D=(v_{n,n})+\sum_{1\leq i,j\leq n}(v_{i,j}).\]
Thinking of $G\cart G$ as an $(n+1)\times (n+1)$ grid, $D$ places one chip on each vertex of the upper left $n\times n$ corner, except on $v_{n,n}$, where it places two chips.  We claim that $r(D)>0$.  To see this, consider firing all vertices $v_{i,j}$ where $i,j\leq n$.  Most of these chip-firing moves cancel, and the net effect is that for all $i\leq n$, a chip moves from the vertex $v_{i,n}$ to the vertex $v_{i,n+1}$, and a chip moves from the vertex $v_{n,i}$ to the vertex $v_{n+1,i}$.  Call this new divisor $D'$.  Then consider firing every vertex except for $v_{n+1,n+1}$; this transforms $D'$ into $D''$, and moves chips from the vertices $v_{n,n+1}$ and $v_{n+1,n}$ to $v_{n+1,n+1}$. These three divisors are illustrated for the case of $n=4$ in Figure \ref{figure:larger_gaps}.

\begin{figure}[hbt]
   		 \centering
\includegraphics[scale=0.8]{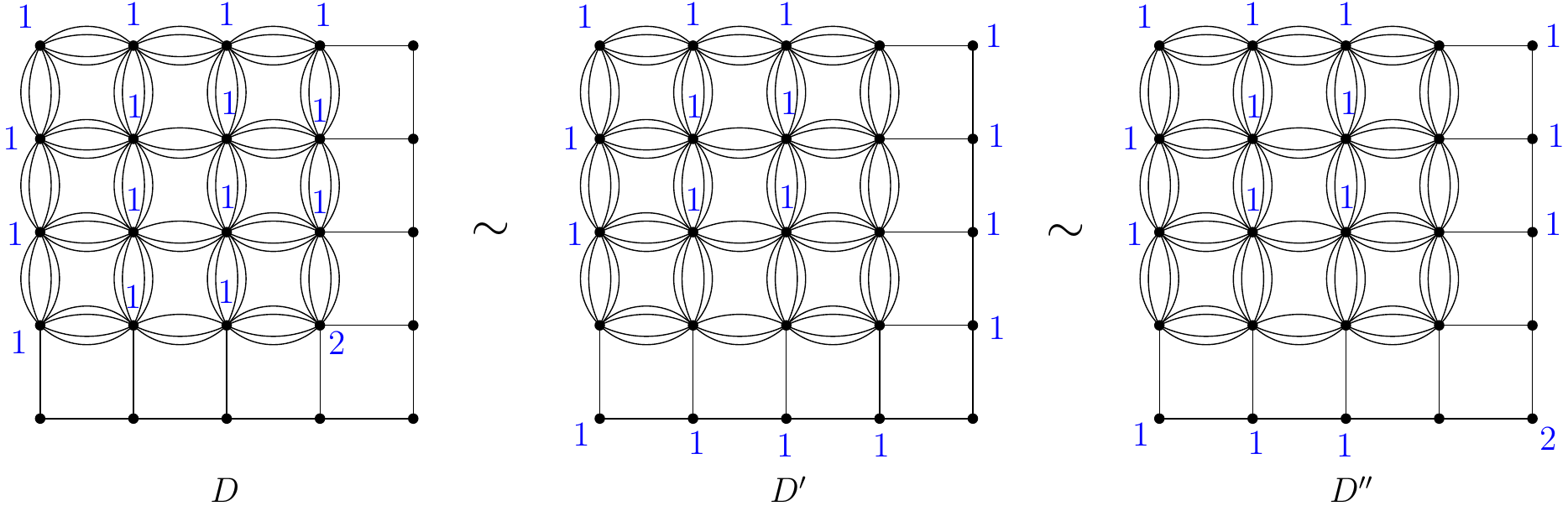}
	\caption{The divisors $D$, $D'$, and $D''$ when $n=4$}
	\label{figure:larger_gaps}
\end{figure}

To see that $r(D)>1$ note that for any vertex $v\in G\cart G$, we have that at least one of $D$, $D'$, and $D''$ places at least one chip on $v$, and that $D$, $D'$, and $D''$ are all effective.  Thus at least one of $D-(v)$, $D'-(v)$, or $D''-(v)$ is effective.  Since all of these are equivalent to $D-(v)$, we conclude that $r(D)>0$.  This means that $\gon(G\cart G)\leq \deg(D)=n^2+1$.  We then have
\[\min\{\textrm{gon}(G)\cdot |V(G)|,\textrm{gon}(G)\cdot |V(G)|\}-\gon(G\cart G)\geq n^2+n-(n^2+1)= n-1,\]
as desired.
\end{proof}

This proof constructed examples of non-simple graph products with lower gonality than expected.  The following example does the same for simple graph products.

\begin{example}\label{example:simple}
Construct $G$ as follows:  start with the complete graph $K_4$, and add four vertices in a path to one vertex.  This graph with labelled vertices, along with a rank one divisor of degree $3$, is pictured on the left in Figure \ref{figure:g_and_g_squared}.  Note that $\gon(G)=3$:  the rank $1$ divisor gives an upper bound, and since $G$ has $K_4$ as a minor it has treewidth at least $3$, which serves as a lower bound on gonality (see Proposition \ref{ref:treewidth}).  The product $G\cart G$ is pictured on the right in Figure \ref{figure:g_and_g_squared}, and its expected gonality is $\gon(G)\cdot |V(G)|=3\cdot 8=24$.

\begin{figure}[hbt]
   		 \centering
\includegraphics[scale=0.8]{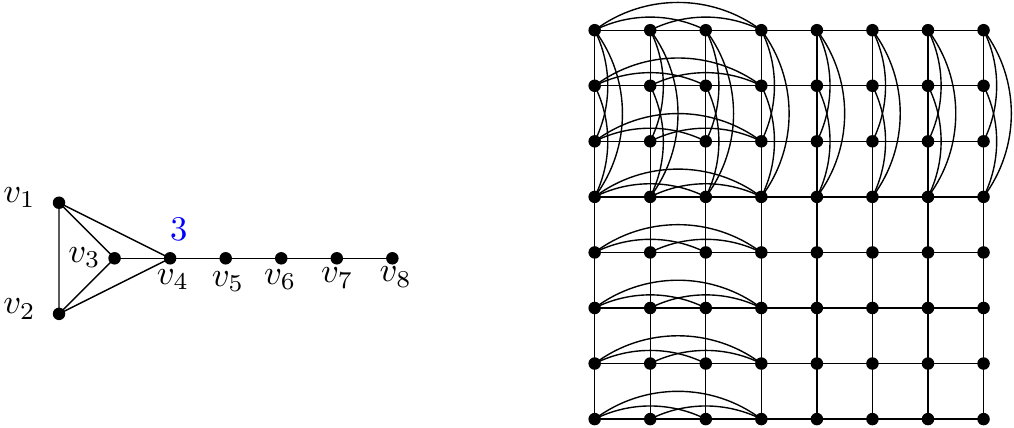}
	\caption{The graph $G$ from Example \ref{example:simple} with a rank $1$ divisor, along with $G\cart G$}
	\label{figure:g_and_g_squared}
\end{figure}

However, the gonality of $G\cart G$ is smaller than $24$.  Let $v_{i,j}$ refer to the vertex $(v_i,v_j)$, and consider
\[D=(v_{4,4})+\sum_{1\leq i, j\leq 4}(v_{i,j})+2(v_{5,5})+2(v_{6,6})+2(v_{7,7}).\]
This divisor is illustrated in the upper left of Figure \ref{figure:counterexample_simple}, where a dotted box indicates $1$ chip on each enclosed vertex.  Note that $\deg(D)=23$.  Chip-firing the upper left $4\times 4$ square of vertices, then the $5\times 5$, then the $6\times 6$, then the $7\times 7$, and finally all vertices but $v_{8,8}$ transforms $D$ iteratively into the other five divisors in Figure \ref{figure:counterexample_simple}.  All six divisors are effective, and together cover each vertex.  As in the previous proof, this implies that $r(D)>0$, so $\gon(G\cart G)\leq 23$, even though the expected gonality is $24$.

\begin{figure}[hbt]
   		 \centering
\includegraphics[scale=0.8]{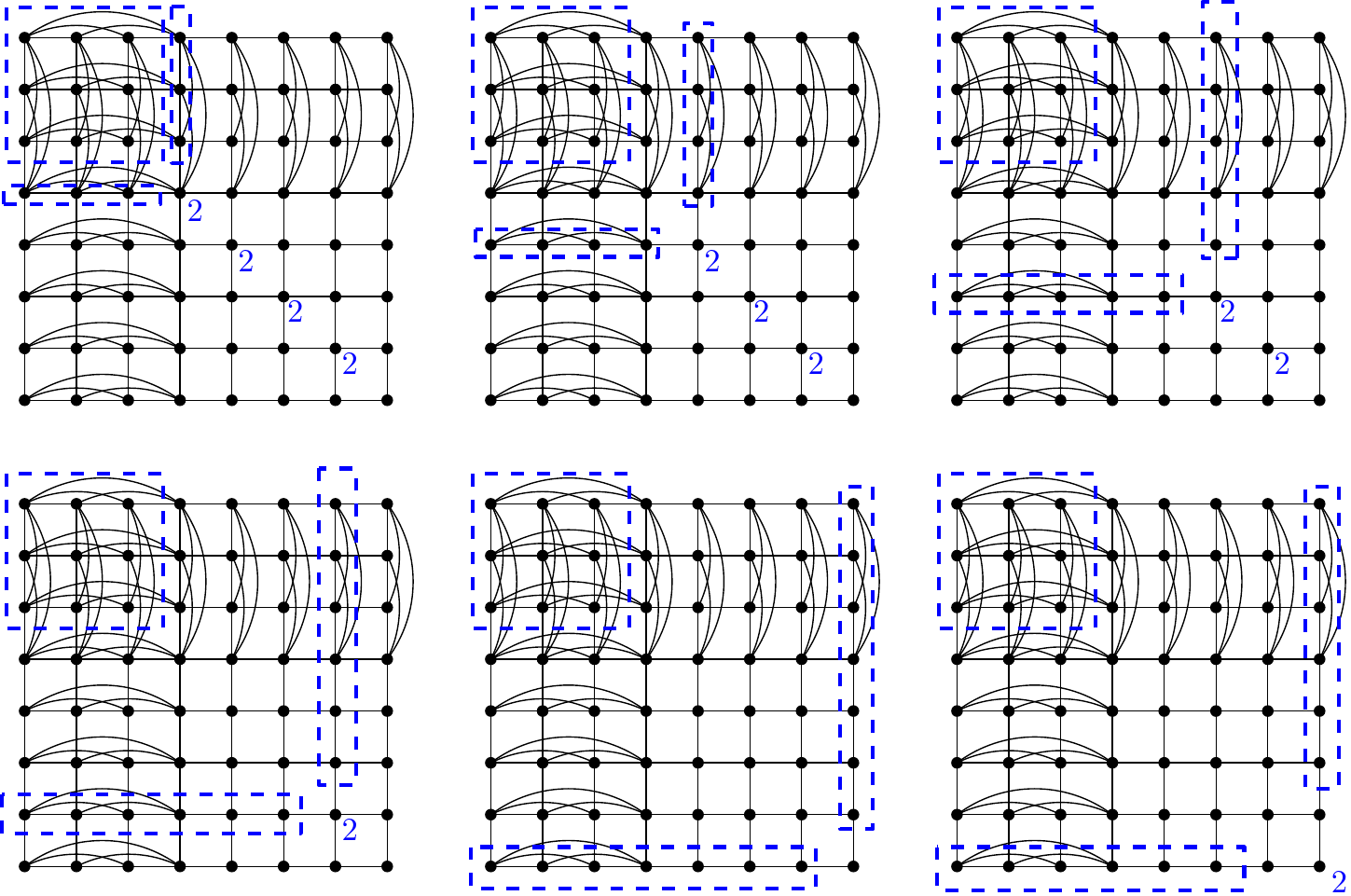}
	\caption{The divisor $D$, and five equivalent divisors; all vertices within a dotted box receive one chip each}
	\label{figure:counterexample_simple}
\end{figure}

This construction can be naturally generalized in two ways.  First, note that we can obtain arbitrarily large gaps between expected and actual gonality by changing $G$ to have more vertices on the path attached to $K_4$.  Each added vertex increases the expected gonality by $3$, but we can get away with adding just $2$ chips to our divisor along the diagonal.

Second, in our construction of $G$ we can replace $K_4$ with any graph $H$ of gonality $\gamma\geq 3$; attaching a path of vertices will not change the gonality.  If $H$ has $m$ vertices and we attach $n$ vertices in a path to obtain $G$, then the expected gonality of $G\cart G$ is $(m+n)\gamma$.  However, mirroring the $D$ construction above gives us a positive rank divisor of degree $m^2+1+2(n-1)=m^2+2n-1$.  Since $\gamma\geq 3$, we can choose $n$ sufficiently large so that the gap $(m+n)\gamma-(m^2+2n-1)=m(\gamma-m)+n(\gamma-2)+1$ is as large as we want.

\end{example}

\section{Graph products with expected gonality}
\label{section:evidence}

In most cases where the gonality of a graph product is known, the inequality from Proposition \ref{prop:upperbound} is in fact an equality.  Let $P_n$ denote the path on $n$ vertices, and let $C_n$ denote the cycle on $n$ vertices.  Since $P_n$ is a tree, $\gon(P_n)=1$; and by Lemma \ref{lemma:g1}, we have $\gon(C_n)=2$.

\begin{itemize}

    \item The \emph{grid graph} $G_{m,n}$ is the product $P_m\boxempty P_n$ of two path graphs. The grid graph has gonality $\min\{m,n\}$ \cite{db}, which is equal to $\min\{|V(P_m)|\gon(P_n),|V(P_n)|\gon(P_m)\}$.
    
    \item The \emph{stacked prism graph} $Y_{m,n}$ is the product $C_m\boxempty P_n$.  For $m\neq 2n$, this graph is known to have gonality $\min\{m,2n\}$ \cite{treewidth}, which is equal to $\min\{|V(C_m)|\gon(P_n),|V(P_n)|\gon(C_m)\}$.
    
    \item  The \emph{toroidal grid graph}  $T_{m,n}$ is the product $C_m\boxempty C_n$ of two cycle graphs.  When $|m-n|\geq 2$, the graph $T_{m,n}$ has gonality $2\min\{m,n\}$ \cite{treewidth}, which is equal to $\min\{|V(C_m)|\gon(C_n),|V(C_n)|\gon(C_m)\}$.
\end{itemize}
In this section we provide some additional instances of $G\cart H$ that satisfy the equation in Question \ref{conjecture:product} by proving that graphs of the form $T\boxempty T'$ and $K_n\boxempty T$ have the expected gonality, where $K_n$ is the complete graph on $n$ vertices and $T$ and $T'$ are trees.  We then prove that the same holds for $K_m\boxempty K_n$ where $\min\{m,n\}\leq 5$, and for graphs of the form $G\cart K_2$ where $g(G)=1$.

The first two arguments involve the  \emph{treewidth} of a graph. We refer the reader to \cite{st} for the definition of treewidth.
In \cite{vddbg}, it was shown that  the {treewidth} of a graph $G$, written $\textrm{tw}(G)$, is a lower bound on its gonality:

\begin{proposition}[Theorem 2.1 in \cite{vddbg}]
\label{ref:treewidth}
For any graph $G$, $\text{gon}(G) \geq \text{tw}(G)$.
\end{proposition}

One way to determine the treewidth of a graph is by use of \emph{brambles}.  A collection $\mathcal{B}=\{B_i\}$ of connected subgraphs of a graph $G$ is called a bramble if every pair of subgraphs $B_i$ and $B_j$ in $\mathcal{B}$ either intersect in a vertex, or contain two vertices that share an edge.  If $B_i\cap B_j$ is nonempty for all $i$ and $j$, then the bramble is called a \emph{strict bramble}.  The \emph{order} a bramble (or a strict bramble) $\mathcal{B}$, denoted $||\mathcal{B}||$, is the cardinality of the smallest collection of vertices $S\subset V(G)$ such that $S\cap B_i$ is nonempty for all $B_i\in V(G)$.  A famous theorem due to Seymour and Thomas says that the treewidth of a graph is equal to one less than the largest order of any bramble of that graph \cite{st}.  By \cite[\S 2]{treewidth}, treewidth is lower bounded by the maximum order of a strict bramble, meaning that the order of any strict bramble is a lower bound on gonality; an earlier proof that strict brambles provide a lower bound on gonality appears in \cite{db}.

\begin{proposition}\label{prop:tree_tree} If $T$ and $T'$ are trees with $m$ and $n$ vertices, respectively, then $\textrm{gon}(T\boxempty T')=\min\{m,n\}$.
\end{proposition}

\begin{proof}  By Proposition \ref{prop:upperbound}, we have $\textrm{gon}(T\boxempty T')\leq \min\{m,n\}$ since the gonality of any tree is $1$.  For a lower bound,
we construct a strict bramble on $T\boxempty T'$ of order $\min\{m,n\}$, thus giving us a lower bound of $\min\{m,n\}$ on its gonality.  For each $v\in V(T)$ and $v'\in V(T')$, include the union  $(\{v\}\boxempty T')\cup (T\boxempty \{v'\})$ in the set $\mathcal{B}$.  Then $\mathcal{B}$ is a strict bramble:  $\{v\}\boxempty T'$ intersects $T\boxempty \{v'\}$ at $(v,v')$, so every element of $\mathcal{B}$ is connected; and any two elements of the bramble intersect, since
\[\left((\{v\}\boxempty T')\cup (T\boxempty \{v'\})\right)\cap\left((\{w\}\boxempty T')\cup (T\boxempty \{w'\})\right)=\{(v,w'),(w,v')\}. \]

Now, let $S\subseteq V(T\boxempty T')$ be a set of size $\min\{m,n\}-1$.  Since there are $m$ pairwise disjoint graphs of the form $\{v\}\boxempty T'$ and $n$ pairwise disjoint graphs of the form $T\boxempty \{v'\}$, at least one of each such graph fails to intersect $S$, meaning that their union, which is an element of $\mathcal{B}$, also fails to intersect $S$.  It follows that $||\mathcal{B}||\geq \min\{m,n\}$.  Thus, $\textrm{gon}(T\boxempty T')\geq\textrm{tw}(T\boxempty T')\geq ||\mathcal{B}||\geq\min\{m,n\}$.  We conclude that $\textrm{gon}(T\boxempty T')=\min\{m,n\}$.
\end{proof}

\begin{proposition}\label{prop:tree_complete} If $T$ is a tree with at least two vertices, then $\textrm{gon}(K_n\boxempty T)=n$.
\end{proposition}

This is a generalization of \cite[Corollary 5.2]{gon3}, which proved $\gon(K_3\cart T)=3$ whenever $T$ is a tree with at least two vertices.

\begin{proof}
We know $\textrm{gon}(K_n\boxempty T)\leq \min\set{(n-1) \cdot \abs{V(T)},1\cdot n}=n$ by Proposition \ref{prop:upperbound}.  By Proposition \ref{ref:treewidth}, it's enough to show that $\textrm{tw}(K_n\boxempty T)\geq n$.  Since $T$ has at least two vertices, we know that $K_n\boxempty K_2$ is a minor (indeed, an induced subgraph) of $K_n\boxempty T$.  The graph $K_n\boxempty K_2$ in turn has $K_{n+1}$ as a minor, obtained from collapsing one copy of $K_n$ to a single vertex.  Since $K_{n+1}$ has treewidth $n$ and treewidth is minor monotonic, we know that $\textrm{tw}(K_n\boxempty T)\geq n$.  This completes the proof.
\end{proof}

\begin{figure}[hbt]
   		 \centering
\includegraphics[scale=1]{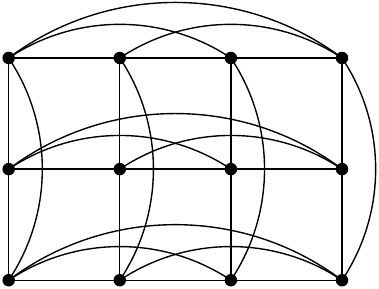}
	\caption{The $3\times 4$ rook's graph}
	\label{figure:rooks_example}
\end{figure}

Our next results involve the $m\times n$ \emph{rook's graph}, which is the product $K_m\boxempty K_n$. The $3\times 4$ rook's graph is pictured in Figure \ref{figure:rooks_example}. If the equation in Question \ref{conjecture:product} is satisfied, the graph $K_m\boxempty K_n$ will have gonality $\min\{(m-1)n,m(n-1)\}$. Unfortunately, to prove this it will not suffice to use treewidth as a lower bound, since for rook's graphs there appears to be a gap between treewidth and gonality.  In the case of the $n\times n$ rook's graph, we have $\textrm{tw}(K_n\boxempty K_n)=\frac{n^2}{2}+\frac{n}{2}-1$ for $n\geq 3$ \cite{lucena}; for large $n$ this is about half as large as the expected gonality of $n(n-1)$.  More generally, it follows from work in \cite{pathwidth} that for $n\geq m\geq 2$, we have
\[\textrm{tw}(K_m\boxempty K_n)\leq\begin{cases}
\frac{m}{2}n+\frac{m}{2}-1 &\textrm{if $m$ is even}\\
\lceil\frac{m}{2}\rceil n -1 &\textrm{if $m$ is odd;}
\end{cases}\] in particular, those authors proved that the above formula computes a number called the \emph{pathwidth} of the rook's graph, and pathwidth serves as an upper bound on treewidth. For large $m$ and $n$, this upper bound on treewidth is about half as large as $(m-1)n$, which is the expected gonality of the $m\times n$ rook's graph with $m\leq n$.

We will make frequent use of the following result, which is a weaker version of Dhar's burning algorithm \cite{dhar}.

\begin{lemma}\label{lemma:burning}
Let $D$ be an effective divisor on a graph $G$, and let $v\in V(G)$ with no chips from $D$.  Let a burning process propagate through $G$ as follows:  set $v$ on fire, and let any edge incident to a burning vertex burn.  If at any point a vertex has more burning edges incident to it than it has chips from $D$, let that vertex burn.  If the whole graph burns, then the debt in $D-(v)$ cannot be eliminated through chip firing, so $r(D)=0$.
\end{lemma}

In some cases this lemma can be used to provide a lower bound on gonality.  We will use it to show the following result, which implies the well-known fact that $\textrm{gon}(K_n)\geq n-1$.

\begin{lemma}\label{lemma:burnination}
Let $K_n$ be a complete graph on $n$ vertices, and let $D$ be an effective divisor of degree at most $n-2$.  Let $v$ be any vertex on which $D$ has no chips.  Then running the burning process from Lemma \ref{lemma:burning} burns the whole graph.
\end{lemma}

\begin{proof}  Suppose for the sake of contradiction that the burning process terminates before the whole graph burns, say with $u$ unburned vertices where $1\leq u\leq n-1$. Each unburned vertex has $n-u$ burning edges incident to it, meaning that there must be at least $u(n-u)$ chips on the graph.  As a function of $u$, the expression $u(n-u)$ is concave down, meaning that it achieves its minimum on the interval $1\leq u\leq n-1$ at a boundary point. (We will frequently make use of such a concavity argument over the next several results.) Plugging in the boundary points $u=1$ and $u=n-1$ yields $1\cdot (n-1)=n-1$ and $(n-1)(n-(n-1))=n-1$, so at minimum there are $n-1$ chips on the graph.  This is a contradiction to $\deg(D)\leq n-2$, so we conclude that the whole graph burns.
\end{proof}

  This lets us prove the following result, which is a step towards computing the gonality of small rook's graphs.  We can describe a vertex of $K_m\boxempty K_n$ as $(i,j)$, where $1\leq i\leq m$ and $1\leq j\leq n$.  The set of all vertices in $K_m\boxempty K_n$ with a fixed $i$ will be called a row, and the set of all vertices in $K_m\boxempty K_n$ with a fixed $j$ will be called a column.  Note that each row is a copy of $K_n$, and each column is a copy of $K_m$.

\begin{lemma}\label{lemma:rooks_lemma}
Let $G=K_m\boxempty K_n$ be a rook's graph where $2\leq m\leq n$, and let $D$ be a divisor on $G$ with $n(m-1)-1$ chips.  Then there exists a vertex $v\in V(G)$ such that running the burning process from Lemma \ref{lemma:burning} starting at $v$ results in at least two whole rows and at least whole two columns of $G$ being on fire.
\end{lemma}

\begin{proof}
By the Pigeonhole principle, one of the $n$ columns must have fewer than $m-1$ chips.  Choose $v$ to be a vertex in this column that has no chips on it, and run the burning process.  Then this whole column burns by Lemma \ref{lemma:burnination}.  Since $n(m-1)-1\leq (n-1)m-1$, one of the $m$ rows must have fewer than $n-1$ chips.  Since a whole column is burning, some vertex in this row is on fire, which by another application of Lemma \ref{lemma:burnination} means that whole row must be on fire.

Suppose for the sake of contradiction that there are no more rows that burn entirely.  An unburned row has $u$ unburned vertices, where $1\leq u\leq n-1$. Each unburned vertex has at least $n-u+1$ burning edges coming into it ($(n-u)$ from the same row, and $1$ from the burning row), meaning that the row has at least $n-u+1$ chips on each unburned vertex.  This means the whole row has $u(n-u+1)$ chips on it. As a function of $u$, this is concave down, and so is minimized on the interval $1\leq u\leq n-1$ at the boundary points.  Plugging in $u=1$ and $u=n-1$ yields $1\cdot(n-1+1)=n$ and $(n-1)(n-(n-1)+1)=2n-2$, respectively; since $n\geq 2$ we have $n\leq 2n-2$, so each unburned row has at least $n$ chips on it.  There are $m-1$ unburned rows, meaning that there must be $(m-1)n$ chips on the graph, a contradiction since we have only placed $n(m-1)-1$ chips.  Thus a second row must burn entirely.

An identical argument shows that if no second column burns entirely, then the graph has at least $(n-1)m$ chips on it; since $(n-1)m\geq (m-1)n$, this yields the same contradiction, so a second column must burn entirely.
\end{proof}

We will also use a result from \cite{spencer}. A \emph{sourceless partial orientation} on a graph $G=(V,E)$ is a choice of orientations on some subset of the edges $E$, such that every vertex has at least one incoming edge.  Given a sourceless partial orientation $\mathcal{O}$, the divisor $D_{\mathcal{O}}$ is the chip configuration with $\textrm{indeg}_\mathcal{O}(v)-1$ chips on the vertex $v$, where $\textrm{indeg}_\mathcal{O}(v)$ is the number of edges oriented towards $v$ in $\mathcal{O}$.

\begin{lemma}\label{lemma:spencer}
Let $D$ be a divisor on $G$ with $\deg(D)\leq g-1$.  Then $r(D)\geq 0$ if and only if $D\sim D_{\mathcal{O}}$ where $\mathcal{O}$ is  a sourceless partial orientation.
\end{lemma}

It follows that given a divisor $D$ of nonnegative rank and degree at most $g-1$, we can find an equivalent divisor $D'$ such that every vertex $v$ has at most $\deg(v)-1$ chips, and such that no two adjacent vertices $v$ and $v'$ have exactly $\deg(v)-1$ and $\deg(v')-1$ chips, respectively.  We are now ready to prove our main theorem regarding rook's graphs.

\begin{theorem}\label{theorem:rooks}
Let $2\leq m\leq n$ and $m\leq 5$.  Then $\textrm{gon}(K_m\boxempty K_n)=(m-1)n$.
\end{theorem}

\begin{proof}  By Proposition \ref{prop:upperbound} we have $\textrm{gon}(K_m\boxempty K_n)\leq \min\{(m-1)n,m(n-1)\}=(m-1)n$.  Thus we must show $\textrm{gon}(K_m\boxempty K_n)\geq (m-1)n$.  This is equivalent to showing that no effective divisor of degree $n(m-1)-1$ has positive rank.  Choose $D$ to be an arbitrary effective divisor of this degree. If $m=2$, choose $v$ according to Lemma \ref{lemma:rooks_lemma}. Starting the burning process at $v$ burns both rows, and thus the entire graph, so the debt cannot be eliminated in $D-(v)$.  It follows that $r(D)<1$, and since $D$ was arbitrary we have  $\textrm{gon}(K_2\boxempty K_n)\geq (2-1)n=n$.  (Since $K_2$ is a tree, this result also follow from Proposition \ref{prop:tree_complete}.)

For $m\geq 3$, we will use the same proof idea, although we will be slightly more careful with our divisor $D$. Using the fact that $n\geq m\geq 3$, we have 
\begin{align*}g(K_m\boxempty K_n)\,=\,&m{\binom{n}{2}}+n{\binom{m}{2}}-mn+1
\\\,\geq\,& mn+mn-mn+1
\\\,=\, & mn-1>n(m-1)-1.
\end{align*}
Thus we have that $\deg(D)\leq g(K_m\boxempty K_n)-1$, meaning we may apply Lemma \ref{lemma:spencer} to find a sourceless partial orientation $\mathcal{O}$ with $D\sim D_{\mathcal{O}}$.  Replace $D$ with this new divisor.  Since the degree of every vertex in the graph is $n+m-2$, we now have that $D$ places no more than $n+m-3$ chips on any vertex, and that no two adjacent vertices both have $n+m-3$ chips.

Our strategy for $m\geq 3$ is as follows: we will show that there exists a vertex $v$ such that, running the burning process from Lemma \ref{lemma:burning} starting at $v$, the entire graph burns.  By that lemma this will imply that the debt in $D-(v)$ cannot be eliminated, so $r(D)<1$.  When $m=5$ we may need to further modify our divisor via chip-firing before we can find this $v$, but we will still obtain the desired result.

For $m=3$ and $m=4$, choose $v$ according to Lemma \ref{lemma:rooks_lemma} based on our divisor $D$. Starting the burning process at $v$, we know by Lemma \ref{lemma:burnination} that at least two rows and two columns will burn.

Let $m=3$, so we have $2n-1$ chips on the graph, and no vertex has more than $n$ chips. We know the whole graph burns except possibly for some vertices in a single row. Suppose for the sake of contradiction that there are $u>0$ unburned vertices. Then we have $u\leq n-2$ since two columns are burning; and we have $2\leq u$ since if $u=1$ we would need $n+1$ chips on the sole unburned vertex. Every unburned vertex must have $n-u+2$ chips on it,  which means there must be at least $u(n-u+2)$ chips on this row, where $2\leq u\leq n-2$; note that these bounds imply that $n\geq 4$.  This number of chips in concave down in $u$, and so obtains its minimum on the interval $2\leq u\leq n-2$ at a boundary point. Thus its minimum on this interval is either $2(n-2+2)=2n$ or $(n-2)(n-(n-2)+2)=4n-8$, and $4n-8\geq 2n$ since $n\geq 4$.  This means we must have more than $2n-1$ chips, a contradiction.  Thus the whole graph burns, and $r(D)<1$.

Let $m=4$.  We have the following: there are $3n-1$ chips on the graph; no vertex has more than $n+1$ chips; and no two adjacent vertices have $n+1$ chips each.  The whole graph burns except possibly for vertices in at most two rows. Suppose for the sake of contradiction that the  whole graph does not burn. We will split into two cases:  where the unburned vertices are all in one row, and where they are spaced out over two rows.

\begin{itemize}

\item \noindent\textbf{Case 4.1}: Suppose there are $u>0$ unburned vertices, all in the same row.  Note that we cannot have $u=1$, since a sole unburned vertex would need $n+2$ chips. We also cannot have $u=2$, since the two unburned vertices would each need $n+1$ chips, but they are adjacent since they are in the same row.  Thus $3\leq u\leq n-2$. It follows that $n-2\geq 3$, so $n\geq 5$.  Each unburned vertex is incident to $n-u+3$ burning edges, so there must be at least $u(n-u+3)$ chips on the graph.  For $3\leq u\leq n-2$ this is minimized either when $u=3$ or when $u=n-2$, which yield $3n$ and $5n-10$.  For $n\geq 5$ these are both strictly larger than $3n-1$, a contradiction.

\item \noindent\textbf{Case 4.2}: Suppose there are unburned vertices in two rows, without loss of generality the first and second rows.  Let $u_1$ and $u_2$ be the number of unburned vertices in these rows, respectively; we have $1\leq u_i\leq n-2$, and without loss of generality $u_1\leq u_2$. Since there are $3n-1$ chips on the graph, we have that the first row has at most $\floor{\frac{3n-1}{2}}$ chips. Each of the $u_1$ unburned vertices in this row has at least $n-u_1+2$ burning edges incident to it, so this row must have at least $u_1\cdot(n-u_1+2)$ chips.  The same argument as in the $m=3$ case shows that for $2\leq u_1\leq n-2$, the minimum value of this function is either $2n$ or $4n-8$.  Both of these exceed $\floor{\frac{3n-1}{2}}$ since $n\geq 4$.  Thus it must be the case that $u_1=1$, so there is a single unburned vertex $v_1$ in this row. Since it did not burn, it must have $n-1+2=n+1$ chips on it, the maximum number allowed.

Consider the second row, which has $u_2$ unburned vertices.  Since there are $n+1$ chips on $v_1$, there are at most $2n-2$ chips on this row.  Each unburned vertex has $n-u_2+3$ burning edges incident to it, except possibly for one in the same column as $v_1$, which would have $n-u_2+2$ burning edges.  This means there are at least $(u_2-1)(n-u_2+3)+(n-u_2+2)=-u_2^2+(n+3)u_2-1$ chips on this row.  Plugging in $u_2=2$ yields $2n+1$, and plugging in $u_2=n-2$ yields $5n-11$.  Both of these are larger than $2n-2$ for $n\geq 4$, and one of them is the minimum value of $-u_2^2+(n+3)u_2-1$ for $2\leq u\leq n-2$; it follows that $u_2=1$.  Call the one unburned vertex in the second row $v_2$. It must have $n+1$ chips on it, and it must be in the same column as $v_1$; otherwise $v_1$ and $v_2$ would each require $n+2$ chips.  But this means we have two adjacent vertices with $n+1$ chips each, which is not allowed in $D$, giving us a contradiction.
\end{itemize}
In both cases, we have reached a contradiction, so the entire graph $K_4\boxempty K_n$ burns. It follows that $r(D)<1$.

Finally, let $m=5$. We have the following: there are $4n-1$ chips on the graph; no vertex has more than $n+2$ chips; and no two adjacent vertices have $n+2$ chips each.  Before choosing $v$ and running the burning process, we will modify our divisor $D$ slightly so that we may make an additional assumption on it:  we would like to be able to assume that no three mutually adjacent vertices have $n+1$ chips each.  Suppose $D$ \emph{does} place $n+1$ chips on three vertices $v_1$, $v_2$, and $v_3$, all mutually adjacent (so either all three are in the same row, or all three are in the same column).  Since there are $4n-1$ chips on the graph, there are $n-4$ chips off of the vertices $v_1$, $v_2$, and $v_3$; in particular, no other vertex has more than $n-4$ chips on it.  Perform three chip-firing moves by firing $v_1$, $v_2$, and $v_3$.  These vertices each end up with $0$ chips; the vertices in the same row or column as all three vertices each gain three chips; and the vertices that share a row or column with a single $v_i$ each gain $1$ chip.  This new divisor has no more than $n-4+3=n-1$ chips on each vertex, and so satisfies 
three conditions:  no vertex has more than $n+2$ chips; no two adjacent vertices have $n+2$ chips each; and no three mutually adjacent vertices have $n+1$ chips each.  Either our starting divisor $D$ satisfied these three conditions, or it failed in the third one and we may replace it with this new divisor.  In any case, we may assume that our divisor $D$ satisfies all three conditions.  We may also assume that at least one of the following two conditions holds: either $D= D_{\mathcal{O}}$ for some sourceless partial orientation; or $D$ has at most $n-1$ chips on each vertex.

As usual, pick $v$ according to Lemma \ref{lemma:rooks_lemma}.  The whole graph burns except possibly for vertices in at most three rows. Suppose for the sake of contradiction the whole graph does not burn.  We will consider three cases, namely when unburned vertices are spread out over one, two, or three rows.

\begin{itemize}

\item \noindent\textbf{Case 5.1}:  Suppose the unburned vertices are all contained in the same row, and let $u$ be the number of unburned vertices.  Each unburned vertex has $n-u+4$ burning edges incident to it, meaning that the row has at least $u(n-u+4)$ chips.   First consider $4\leq u\leq n-2$, so $n\geq 6$. The function $u(n-u+4)$ on this interval is minimized at an endpoint, and so has minimum $\min\{6n-12,4n\}$; this is larger than $4n-1$ for $n\geq 6$, so we know that $u\leq 3$.  If $u=1$, then the one unburned vertex must have $n+3$ chips, which is not allowed.  If $u=2$, then the two unburned vertices (which are adjacent) must each have $n+2$ chips, which is not allowed.  If $u=3$, then the three unburned vertices (which are mutually adjacent) must each have $n+1$ chips, which is not allowed.  Thus we have reached a contradiction.

\item \noindent\textbf{Case 5.2}:  Suppose the unburned vertices are spread out over the first two rows. Let $u_1$ and $u_2$ denote the number of unburned vertices in these rows, where $u_1\leq u_2$.

The first row has at most $\floor{\frac{4n-1}{2}}=2n-1$ chips. Each unburned vertex in the first row has at least $n-u_1+3$ incident burning edges, so the first row has at least $u_1(n-u_1+3)$ chips.  For $2\leq u_1\leq n-2$, this has minimum $\min\{2n+2,5n-10\}$, which is greater than $2n-1$ for $n\geq 5$.  Thus $u_1=1$.  The one unburned vertex $v_1$ must have $n-1+3=n+2$ chips on it (at least $n+2$ so as not to burn, and at most $n+2$ since this is the maximum allowed).

The second row then has at most $4n-1-(n+2)=3n-3$ chips.  The number of  burning edges incident to unburned vertices in this row is at least $u_2(n-u_2+4)-1$, meaning the row has at least that many chips.  For $3\leq u_2\leq n-2$, this has minimum $\min\{3n+2,6n-13\}$, which is larger than $3n-3$ for $n\geq 5$.  Thus we have $u_2=1$ or $u_2=2$. In either case, $v_1$ shares a column with an unburned vertex in the second row; otherwise $v_1$ would need $n-1+4=n+3$ chips, which is not allowed.

If $u_1=1$, call the unburned vertex in the second row $v_2$.  Then $v_1$ and $v_2$ need $n+2$ chips each so as not to burn, but this is not allowed since they are adjacent to each other.  If $u_2=2$, call the unburned vertices in the second row $v_2$ and $v_3$, where $v_2$ shares a column with $v_1$.
The vertices $v_1$, $v_2$, and $v_3$ must have at least $n+2$, $n+1$, and $n+2$ chips, respectively.  Either $D$ is of the form $D_{\mathcal{O}}$ for some sourceless partial orientation $\mathcal{O}$; or it is a divisor with at most $n-1$ vertices on each vertex.  The second clearly does not hold, and so $D=D_{\mathcal{O}}$.  We know that $\mathcal{O}$ must have every neighbor of $v_1$ and $v_3$ oriented towards them; but this means that at most $n+3-2=n+1$ of $v_2$'s neighbors can be oriented to it, meaning it can have at most $n$ chips, a contradiction.

\item \noindent\textbf{Case 5.3}:  Suppose the unburned vertices are spread out over the first three rows. Let $u_1$, $u_2$, and $u_3$ denote the number of unburned vertices in these rows, where $u_1\leq u_2\leq u_3$. Nearly identical arguments to those from Case 4.2 show that each row has exactly one unburned vertex, call them $v_1$, $v_2$, and $v_3$.  Each has at least $(n-1)+2=n+1$ incident burning edges, so each has at least $n+1$ chips.  If any of $v_1$, $v_2$, or $v_3$ is in its own column, it would need to have $n+3$ chips not to burn, which is more than is allowed on a single vertex.  Thus all three must be in a single column.  But we cannot have three mutually adjacent vertices each with $n+1$ chips, a contradiction.

\end{itemize}
In every one of our three cases, we have reached a contradiction.  Thus the entire graph $K_5\boxempty K_n$ burns, and so $r(D)<1$.  This completes the proof.
\end{proof}

In the course of this proof, we have showed that given a divisor $D$ of degree $n(m-1)-1$ coming from a sourceless partial orientation on $K_m\boxempty K_n$ (where $m\leq n$ and $m\leq 4$), we can choose a vertex $v$ such that running the burning process from Lemma \ref{lemma:burning} starting at $v$ makes the whole graph burn. However, as hinted at by our careful pre-processing of $D$ when $m=5$, this is \textbf{not} true for all rook's graphs.  Consider the divisor of degree $(5-1)\cdot 5-1=19$ pictured on the rook's graph $K_5\boxempty K_5$ on the left in Figure \ref{figure:5x5} (only the vertices of the graph are illustrated).  This divisor does arise from a sourceless partial orientation, namely the partial orientation with the directed edges pictured in the middle and right images in Figure \ref{figure:5x5}; the oriented edges are spread out over two copies of the graph for visibility.  However, no matter our choice of vertex $v$, the whole graph does not burn when we start the burning process from $v$, since the three vertices with $6$ chips each remain unburned.  Although our proof manages to fill in this gap for $K_5\boxempty K_n$, the combinatorics becomes more complicated as we increase $m$,  and so new techniques will need to be developed to push our results for rook's graphs further.

\begin{figure}[hbt]
   		 \centering
\includegraphics[scale=1]{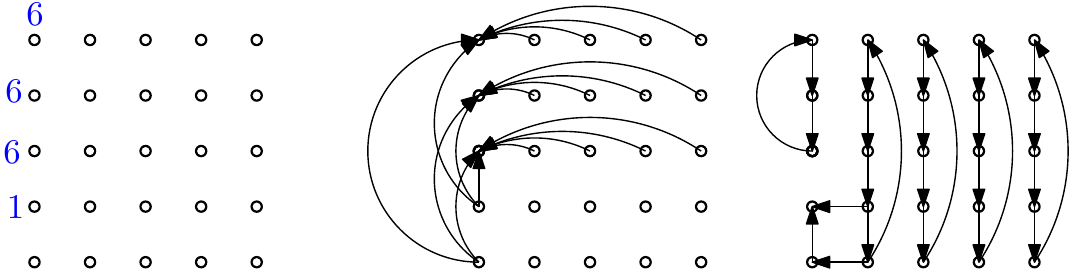}
	\caption{A divisor $D$ on $K_5\boxempty K_5$, and a sourceless partial orientation $\mathcal{O}$ such that $D=D_{\mathcal{O}}$}
	\label{figure:5x5}
\end{figure}

The final result of this section determines the gonality of products of the form $G\cart K_2$, where $G$ is a graph of genus $1$.

\begin{theorem}\label{theorem:genus_1_345}
If $G$ is a graph of genus $1$, then $\textrm{gon}(G\boxempty K_2)=\min\{|V(G)|,4\} $.
\end{theorem}

 We remark that in some instances this result follows from a treewidth argument, such as with $C_5\cart K_2$; but not in others, such as with $C_4\cart K_2$.  Our proof will work in all cases.

\begin{proof}
 Any genus $1$ graph has gonality $2$ by Lemma \ref{lemma:g1}, so by Proposition \ref{prop:upperbound} we have $\textrm{gon}(G\boxempty K_2)\leq \min\{|V(G)|\cdot 1,2\cdot 2\}=\min\{|V(G)|,4\}$.  It remains to show that $\textrm{gon}(G\boxempty K_2)\geq\min\{|V(G)|,4\} $.
 
 First assume that $|V(G)|\leq 3$.  There are three possibilities for $G$:  either $G=K_3$; or $G$ has two vertices, joined by two edges; or $G$ is the previous graph with a third vertex attached by a single edge to one of the other two vertices.  If $G=K_3$, then $\gon(K_3\cart K_2)=3$ by Proposition \ref{prop:tree_complete}.  If $G$ has two vertices joined by two edges, then $\gon(G\cart K_2)\geq 2$, since $G\cart K_2$ is not a tree; thus $\gon(G\cart K_2)=2$.  Finally, if $G$ is the third possible graph, suppose for the sake of contradiction that $\gon(G\cart K_2)=2$.  We illustrate all effective divisors of degree $2$ on $G\cart K_2$ in Figure \ref{figure:list_of_divisors}, up to the natural symmetry of switching  the two copies of $G$. For one member $D$ of each equivalence class we label a vertex $v$ such that the debt in $D-(v)$ cannot be eliminated according to Lemma \ref{lemma:burning}, a contradiction to $\gon(G\cart K_2)=2$.  Thus $\gon(G\cart K_2)=3=|V(G)|$.

\begin{figure}[hbt]
   		 \centering
\includegraphics[scale=0.8]{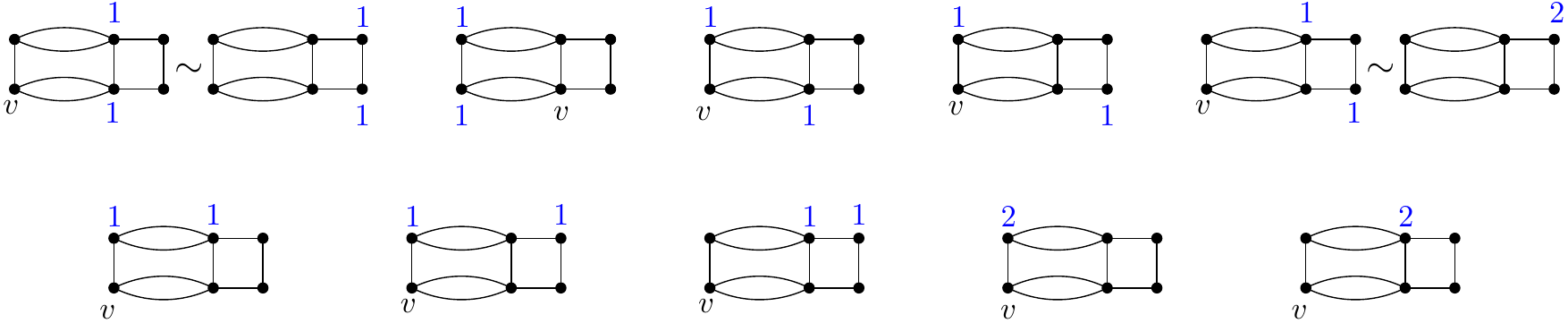}
	\caption{Divisors on $G\cart K_2$ of degree $2$, none have which have positive rank}
	\label{figure:list_of_divisors}
\end{figure}

Having dealt with the case of $|V(G)|\leq 3$, we now assume that $|V(G)|\geq 4$.  We may view $G\boxempty K_2$ as the union of two copies of $G$, with matching vertices connected by edges; we will refer to these two copies of $G$ as $G'$ and $G''$, and any vertices $v'$ and $v''$ corresponding to the same vertex $v$ of $G$ as \emph{parallel vertices}.  Let $C$ denote the unique cycle of $G$, and let $C'$ and $C''$ denote the corresponding cycles in $G'$ and $G''$, respectively.

Suppose for the sake of contradiction that there exists an effective divisor $D$ on $G\boxempty K_2$ of degree $3$ with positive rank.  Since $D$ has positive rank, we may assume that $D$ places at least one chip on  $C'\cup C''$ (indeed, on any vertex of $C'\cup C''$ we choose).  Then, if $D$ has two chips on any $2$-valent vertex, fire that vertex; this will not result in another $2$-valent vertex having two chips, since the two chips go to different vertices and the third chip is on a vertex of $C'\cup C''$, all of whose vertices have valence greater than 2.  After this, suppose that $D$ puts at least one chip on both of two parallel vertices $v'$ and $v''$ corresponding to a vertex $v$ not on the cycle $C$ of $G$. Then we may chip-fire those two vertices, together with all vertices not in the same component of $(G\boxempty K_2)-\{v',v''\}$ as the cycles $C'$ and $C''$.  Applying this process enough times, we will move the two chips to $C'\cup C''$.  This process is illustrated in Figure \ref{figure:moving_parallel}. Finally, if at this point any vertex of $C'\cup C''$ with valence $3$ has exactly $3$ chips, we will fire that vertex; this does not interfere with any of our other assumed properties, since any $3$-valent vertex on $C'\cup C''$ is incident only to other vertices on $C'$ and $C''$. Thus we may assume that $D$ places at least one chip  on $C'\cup C''$; that $D$ does not place $2$ chips on a vertex of valence $2$;  that no two parallel vertices away from $C'\cup C''$ both have a chip; and that no vertex of $C'\cup C''$ of valence $3$ has $3$ chips.

\begin{figure}[hbt]
   		 \centering
\includegraphics[scale=0.8]{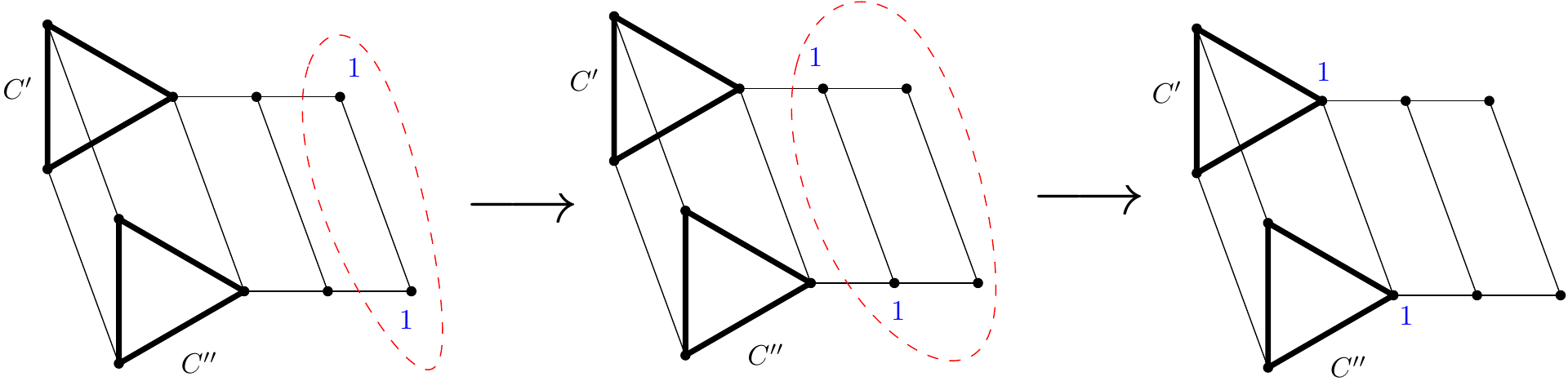}
	\caption{Moving chips on parallel vertices towards $C'$ and $C''$; chip-firing the circled vertices yields the next configuration}
	\label{figure:moving_parallel}
\end{figure}

Since $\deg(D)=3$, one of $G'$ and $G''$ has at most one chip on it; say it is $G'$.  Choose any vertex on $C'$ that does not have a chip on it, and run the burning process from Lemma \ref{lemma:burning} starting from that vertex.  Since $C'$ has at most one chip on it, the whole cycle $C'$ will burn.  We now deal with two cases:  where $C''$ burns solely based on $C'$ burning, and when it does not.  For the moment, assume that $G$ is a simple graph.

\begin{itemize}
    \item Assume $C''$ burns.  Let $v'$ and $v''$ be parallel vertices not on $C'\cup C''$.  They cannot both have a chip, so once both of them receive a burning edge (besides $v'v''$), then one burns and the second will burn unless it has $2$ chips on it. Thus fire will spread through the whole graph, unless some vertex off of $C'$ and $C''$ has $2$ chips on it. If such a vertex exists, call it $v''$, and note that $v''$ cannot be $2$-valent since it has two chips. There is one chip on $C'\cup C''$, say on the vertex $w$.  Since $\deg(D)=3$, there are no chips off of $v''$ and $w$.  Note that $(G\boxempty K_2)-\{v''\}$ is connected, and since $w$ burns, so does all of $(G\boxempty K_2)-\{v''\}$. Then we know $v''$ will burn as well:  it has $2$ chips, and at least $3$ incident burning edges, as illustrated in Figure \ref{figure:tadpole_exception}.  Thus the whole graph burns.

\begin{figure}[hbt]
   		 \centering
\includegraphics[scale=0.8]{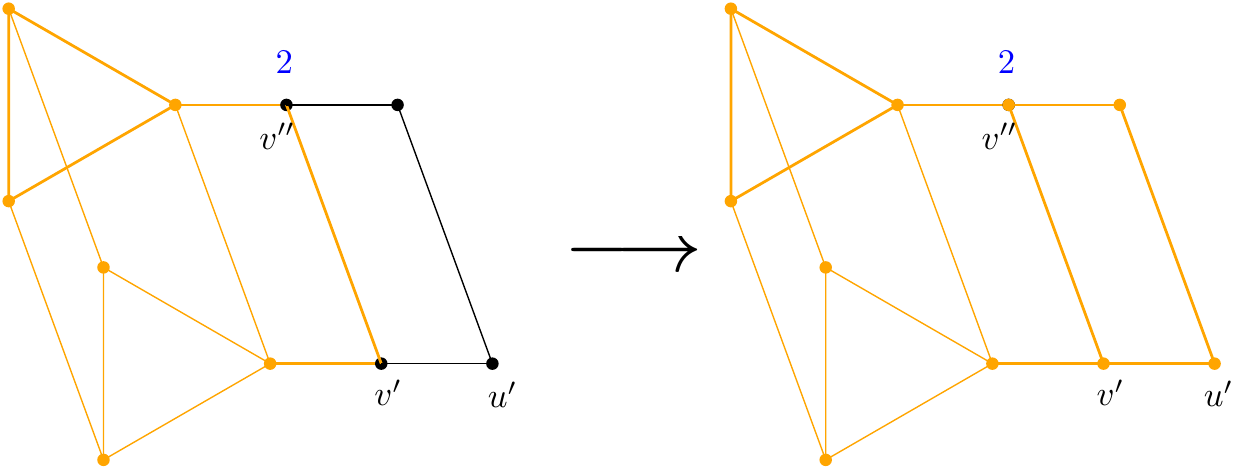}
	\caption{Even with $2$ chips on $v''$, the whole graph eventually burns.  It is important we've assumed a chip is on $C'\cup C''$: otherwise $u'$ could have a chip and the fire could be blocked.}
	\label{figure:tadpole_exception}
\end{figure}

    \item  Assume $C''$ does not immediately burn based on $C'$ burning.   We claim that then all $3$ chips must be on $C''$, and that $G$ is not simply a cycle.  If every vertex of $C''$ has a chip, then indeed $C''$ has all three chips (and we know that $C$ is a triangle, meaning $G\supsetneq C$ since $G$ has at least $4$ vertices).  If on the other hand at least one vertex of $C''$ lacks a chip, then that vertex burns due to an edge from $C'$, and the fire will spread both ways around $C''$, starting from that vertex.  Any configuration short of having $3$ chips on a single vertex result in all of $C''$ burning. Thus since $C''$ does not burn, we know there must be $3$ chips on a vertex $v''$ in $C''$. Since $D$ does not place $3$ chips on any $3$-valent vertex of $C'\cup C''$, we know that $v''$ has valence at least $4$, so it corresponds to a vertex $v$ in $G$ of valence at least $3$; it follows that $G\supsetneq C$.

    Note that $(G\boxempty K_2)-C''$ is connected.  Since $C'$ is on fire, and since there are no chips off of $C''$, all of $(G\boxempty K_2)-C''$ burns.  If each vertex in $C''$ has a chip, then some vertex has an additional burning edge coming from a vertex in $G''-C''$, so that vertex in $C''$ will burn, and from there all $C''$ burns.  If a single vertex $v''$ in $C''$ has $3$ chips, then at this point the whole graph except for $v''$ is burning; as $\deg(v'')\geq 4$, this vertex burns as well.
\end{itemize}
In both cases, the entire graph burns, which by Lemma \ref{lemma:burning} contradicts $r(D)>0$.  

Now assume that $G$ is not a simple graph. It follows that $C$ must be a cycle of length $2$ connecting two vertices.  The argument from the first case carries through when $G$ is not simple.  The second case falls through when $G$ is not simple precisely when the three chips are placed as follows: if there is one chip on each vertex $u''$ and $w''$ of the cycle $C''$, and one chip on a vertex vertex $v'$ on $G'$ that is not on $C'$, but is instead incident to a vertex of $C'$.  Indeed, the whole graph may not burn based on this divisor $D$. If we have this divisor $D$ on $G\cart K_2$, choose a vertex $v$ of $G\cart K_2$  in the following way:  if  $(G\cart K_2)-\{v',v''\}$ has vertices off of $ C''$ incident to $v''$, choose $v$ to be such a vertex; if no such vertex exists, then since $|V(G)|\geq 4$ we may choose $v$ off of $C'\cup C''$ on a component of $(G\cart K_2)-(C'\cup C'')$ not containing $v''$.  The first case leads to $v'$ and then all of $C''$ burning, and the second case leads to all of $C'\cup C''$ and then $v'$ burning.  These two cases are illustrated in Figure \ref{figure:g1_not_simple}.  Thus we have reached our desired contradiction for non-simple graphs as well.
\end{proof}

\begin{figure}[hbt]
   		 \centering
\includegraphics[scale=1]{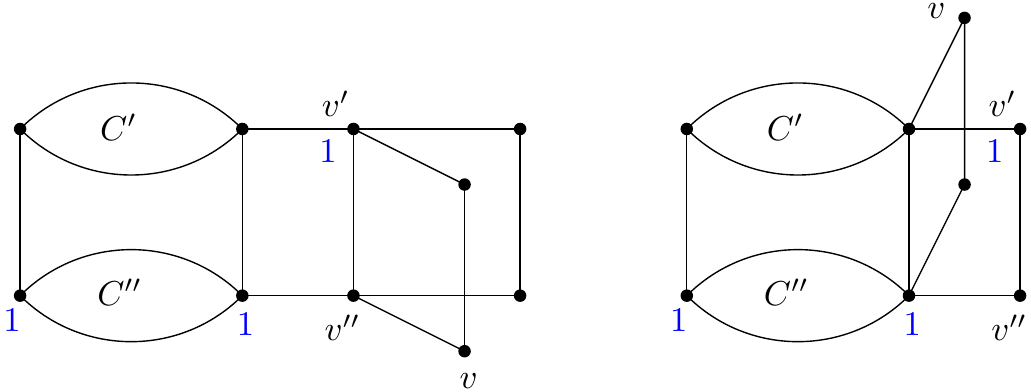}
	\caption{Two special cases for choosing $v$ when $G$ is not simple}
	\label{figure:g1_not_simple}
\end{figure}

We close this section by remarking that there are several open conjectures that would follow if the products in question have the expected gonality from Question \ref{conjecture:product}: that the gonality of the stacked prism $Y_{2n,n}=P_{n}\boxempty C_{2n}$ is $2n$, and the gonality of the toroidal grids $T_{n,n}=C_n\cart C_n$ and $T_{n,n+1}=C_n\cart C_{n+1}$ is $2n$ \cite{treewidth}; that the gonality of the $m \times n \times l$ grid $P_m\cart P_n\boxempty P_l$ is $mnl/\max\set{m,n,l}$  \cite{db}; and that the gonality of the $n$-dimensional cube $Q_n=(K_2)^{\boxempty n}$ is $2^{n-1}$  \cite{vddbg}.  In fact, most of these conjectures would follow from a weaker version of the result, namely that the equation in  Question \ref{conjecture:product} holds if one of the graphs is a path.

\section{The gonality conjecture for products of graphs}
\label{section:gonalityconjecture}

Using our upper bound from Proposition \ref{prop:upperbound}, we will show in this section that  $G\boxempty H$ satisfies the inequality in Conjecture \ref{conjecture:gonality} for any graphs $G$ and $H$ with two or more vertices each.  We first state the following useful lemma.

\begin{lemma}\label{lemma:gon_g}
Let $G$ be a graph of genus $g$.  Then $\textrm{gon}(G)\leq g+1$.  Moreover, if $g\geq 2$, then $\textrm{gon}(G)\leq g$.
\end{lemma}

\begin{proof}
For the first claim we use the following standard Riemann-Roch argument.  Let $D$ be any effective divisor of degree $g+1$ on $G$.  Since $r(K-D)\geq -1$,  Theorem \ref{theorem:rr} tells us that $r(D)+1\geq r(D)-r(K-D)=\deg(D)+1-g=g+1+1-g=2$, so $r(D)\geq 1$.  This means $\textrm{gon}(G)\leq \deg(D)=g+1$.

Now assume $g\geq 2$.  By \cite{bn}, the  divisor $K$ has degree $2g-2$ and rank $g-1$.  By \cite[Lemma 2.7]{baker}, given a divisor $D$ of rank $r\geq 0$, for any vertex $v$ the divisor $D-(v)$ has rank at least $r-1$.  Since $g\geq 2$, we can iteratively subtract $g-2$ vertices from $K$ to obtain a divisor of degree $2g-2-(g-2)=g$ that has rank at least $g-1-(g-2)=1$.  It follows that $\textrm{gon}(G)\leq g$.
\end{proof}

 It will also be helpful to have some notation for non-simple graphs with $2$ or $3$ vertices.  Any such graph must have $K_2$, $P_3$, or $K_3$ as its underlying simple graph.  For $n\geq 2$, the \emph{banana graph} $B_n$ is the graph with two vertices and $n$ edges between them.  For $m,n\geq 1$ with $\max\{m,n\}\geq 2$, the \emph{double banana graph} $B_{m,n}$ is the graph with three vertices, the first two connected by $m$ edges and the second two connected by $n$ edges.  For  For $\ell,m,n\geq 1$ with $\max\{\ell,m,n\}\geq 2$, the \emph{banana loop graph} $L_{\ell,m,n}$ is a graph with three vertices, where the numbers of edges between the three pairs of vertices are $\ell$, $m$, and $n$.  Several examples of these graphs are illustrated in Figure \ref{figure:multigraphs}, along with divisors of rank $1$.  One can verify that $\gon(B_n)=2$; that $\gon(B_{m,n})=2$ if $\min\{m,n\}=1$ or $m=n=2$, and $\gon(B_{m,n})=3$ otherwise; and that $\gon(L_{\ell,m,n})=2$ if two of $\ell,m,$ and $n$ are equal to $1$, and $\gon(L_{\ell,m,n})=3$ otherwise.  We have already seen $B_{2,1}\cart B_{2,1}$ in Figure \ref{figure:counterexample}, which illustrated that $\gon(B_{2,1}\cart B_{2,1})\leq 5$.

\begin{figure}[hbt]
   		 \centering
\includegraphics[scale=1]{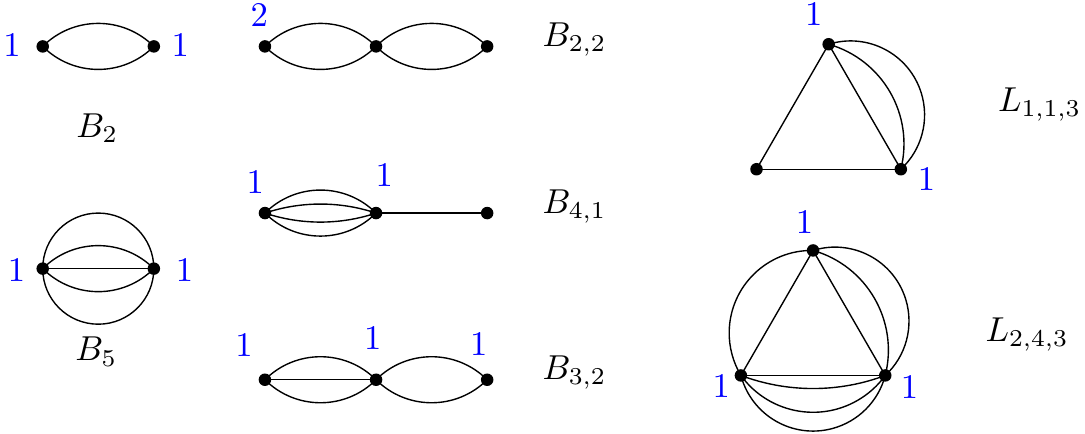}
	\caption{Several multigraphs with at most $3$ vertices, each with a divisor of rank $1$}
	\label{figure:multigraphs}
\end{figure}

We are now ready to prove that any nontrivial graph product $G\boxempty H$ satisfies $\textrm{gon}(G\boxempty H)\leq\floor{\frac{g(G\boxempty H)+3}{2}}.$

\begin{proof}[Proof of Theorem \ref{theorem:main}]
Let $G$ and $H$ be graphs, where $G$ has $v_1$ vertices and $e_1$ edges, and  $H$ has $v_2$ vertices and $e_2$ edges, where $v_1,v_2\geq 2$. The product graph $G\boxempty H$ then has genus $e_1v_2+e_2v_1-v_1v_2+1$. Without loss of generality we will assume that $e_2v_1\leq e_1v_2$, which implies that $e_2v_1\leq \frac{e_1v_2+e_2v_1}{2}$.

Assume for the moment that  either $v_2\geq 4$, or $g(H)\geq 2$.  By Proposition \ref{prop:upperbound} we have $\textrm{gon}(G\boxempty H)\leq v_1\textrm{gon}(H)$.  By Lemma \ref{lemma:gon_g} we have $v_1\textrm{gon}(H)\leq v_1\left(g(H)+1\right)=v_1(e_2-v_2+2)=e_2v_1-v_1v_2+2v_1$.  If $v_2\geq 4$, then we have
\begin{align*}
    \frac{g(G\boxempty H)+3}{2}-\textrm{gon}(G\boxempty H)\,\geq\, &\frac{g(G\boxempty H)+3}{2}-(e_2v_1-v_1v_2+2v_1)
    \\\,=\,& \frac{e_1v_2+e_2v_1-v_1v_2+4}{2}-e_2v_1+v_1v_2-2v_1
    \\\,=\,& \left( \frac{e_1v_2+e_2v_1}{2}-e_2v_1\right)+\frac{v_1v_2}{2}-2v_1+2
    \\\,\geq\,& \frac{v_1v_2}{2}-2v_1+2
    \\\,=\,&v_1\left(\frac{v_2}{2}-2\right)+2
    \\\,\geq\,&v_1\left(\frac{4}{2}-2\right)+2= 2.
\end{align*}  On the other hand, if $g(H)\geq 2$, then by Lemma \ref{lemma:gon_g} we have $\gon(H)\leq g(H)$.  It follows that $\gon(G\cart H)\leq v_1\gon(H)\leq v_1g(H)=e_2v_1-v_1v_2+v_1$.  Thus we have
\begin{align*}
    \frac{g(G\boxempty H)+3}{2}-\textrm{gon}(G\boxempty H)\,\geq\, & \frac{e_1v_2+e_2v_1-v_1v_2+4}{2}-e_2v_1+v_1v_2-v_1
    \\\,=\,& \left( \frac{e_1v_2+e_2v_1}{2}-e_2v_1\right)+\frac{v_1v_2}{2}-v_1+2
    \\\,\geq\,& \frac{v_1v_2}{2}-v_1+2
    \\\,=\,&v_1\left(\frac{v_2}{2}-1\right)+2
    \\\,\geq\,&v_1\left(\frac{2}{2}-1\right)+2= 2.
\end{align*}
In both of these cases, we have that $\textrm{gon}(G\boxempty H)<\frac{g(G\boxempty H)+3}{2}$, and in fact that $\textrm{gon}(G\boxempty H)<\lfloor \frac{g(G\boxempty H)+3}{2}\rfloor$, since the gap between $\textrm{gon}(G\boxempty H)$ and $\frac{g(G\boxempty H)+3}{2}$ is at least $2$.

We may now assume $v_2<4$, and that $g(H)\leq 1$. We might be tempted to say that by symmetry, $v_1<4$ as well; however, we already used the symmetry of switching $G$ and $H$ when we assumed $e_2v_1\leq e_1v_2$, so we have no control on $v_1$ at the moment. Since $v_2<3$, we know that $H$ is $K_2$, $P_3$, $K_3$, $B_n$, $B_{m,n}$, or $L_{\ell,m,n}$ for some integers $\ell$, $m$ and $n$. Note that $g(B_n)=n-1$, $g(B_{m,n})=m+n-2$, and $L_{\ell,m,n}=\ell+m+n-2$.  Since we've assumed $g(H)\leq 1$, the only non-simple possibilities for $H$ are $B_2$ and $B_{2,1}$.  Thus we have $H\in \{K_2,P_3,K_3,B_2,B_{2,1}\}$.  We will handle $K_3$ and $B_{2,1}$ together, and the other three cases separately.

Let $H=K_2$. Then $v_2=2$ and $e_2=1$. We have $g(G\boxempty K_2)=e_1v_2+e_2v_1-v_1v_2+1=2e_1+v_1-2v_1+1=2e_1-v_1+1$, so  $\frac{g(G\boxempty K_2)+3}{2}=e_1-\frac{v_1}{2}+2$.  By Proposition \ref{prop:upperbound} we have $\textrm{gon}(G\boxempty K_2)\leq\min\{v_1,2\textrm{gon}(G)\}$.  From the bound $\textrm{gon}(G\boxempty K_2)\leq v_1$, we deduce
\begin{align*}
    \frac{g(G\boxempty K_2)+3}{2}-\textrm{gon}(G\boxempty K_2)\,\geq\, & e_1-\frac{v_1}{2}+2-v_1
    \\\,=\, & e_1-\frac{3v_1}{2}+2.
\end{align*}
From the bound $\textrm{gon}(G\boxempty K_2)\leq 2\textrm{gon}(G)\leq 2(g(G)+1)=2e_1-2v_1+4$, we deduce
\begin{align*}
    \frac{g(G\boxempty K_2)+3}{2}-\textrm{gon}(G\boxempty K_2)\,\geq\, & e_1-\frac{v_1}{2}+2-2e_1+2v_1-4
    \\\,=\, & -e_1+\frac{3v_1}{2}-2.
\end{align*}
At least one of $e-\frac{3v_1}{2}+2$ and $-e+\frac{3v_1}{2}-2$ is nonnegative, so we have   $ \frac{g(G\boxempty K_2)+3}{2}-\textrm{gon}(G\boxempty K_2)\geq 0$.  Since $\floor{0}=0$ and since $ \left\lfloor\frac{g(G\boxempty K_2)+3}{2}-\textrm{gon}(G\boxempty K_2)\right\rfloor= \left\lfloor\frac{g(G\boxempty K_2)+3}{2}\right\rfloor-\textrm{gon}(G\boxempty K_2)$, we have $\textrm{gon}(G\boxempty K_2)\leq\floor{\frac{g(G\boxempty K_2)+3}{2}}$.

Let $H=P_3$,  so that $v_2=3$ and $e_2=2$. If $G$ is a tree, then since $P_3$ is a tree we know by Proposition \ref{prop:tree_tree} that $\textrm{gon}(G\boxempty P_3)=\min\{3,v_1\}$. Since any tree has one more vertex than it has edges, we then have $\left\lfloor\frac{g(G\boxempty P_3)+3}{2}\right\rfloor=\left\lfloor\frac{e_1v_2+e_2v_1-v_1v_2+4}{2}\right\rfloor=\left\lfloor\frac{3(v_1-1)+2v_1-3v_1+4}{2}\right\rfloor=\left\lfloor\frac{2v_1+1}{2}\right\rfloor=v_1\geq\min\{3,v_1\}$, so the gonality conjecture holds.  Thus we may assume that $G$ is not a tree.   We have $g(G\boxempty P_3)=e_1v_2+e_2v_1-v_1v_2+1=3e_1+2v_1-3v_1+1=3e_1-v_1+1$, so    $\frac{g(G\boxempty P_3)+3}{2}=\frac{3e_1}{2}-\frac{v_1}{2}+2$.  By Proposition \ref{prop:upperbound} we have $\textrm{gon}(G\boxempty P_3)\leq\min\{v_1,3\textrm{gon}(G)\}\leq v_1$.   
Then we have
\begin{align*}
    \frac{g(G\boxempty P_3)+3}{2}-\textrm{gon}(G\boxempty P_3)\,\geq\, & \frac{3e_1}{2}-\frac{v_1}{2}+2-v_1
    \\\,\geq\, & \frac{3}{2}(e_1-v_1)+2.
\end{align*} 
Since $G$ is not a tree, we know $e_1-v_1\geq 0$, so we have $\frac{g(G\boxempty P_3)+3}{2}-\textrm{gon}(G\boxempty P_3)\geq 2$; it follows that $\textrm{gon}(G\boxempty P_3)<\floor{\frac{g(G\boxempty P_3)+3}{2}}$.

Let $H=K_3$, so that $v_2=3$ and $e_2=3$.  We have $g(G\boxempty K_3)=e_1v_2+e_2v_1-v_1v_2+1=3e_1+3v_1-3v_1+1=3e_1+1$, so   $\frac{g(G\boxempty K_3)+3}{2}=\frac{3e_1}{2}+2$.  By Proposition \ref{prop:upperbound} we have $\textrm{gon}(G\boxempty K_3)\leq\min\{2v_1,3\textrm{gon}(G)\}$. 
The bound $\textrm{gon}(G\boxempty K_3)\leq 2v_1$ implies
\begin{align*}
    \frac{g(G\boxempty K_3)+3}{2}-\textrm{gon}(G\boxempty K_3)\,\geq\, & \frac{3e_1}{2}+2-2v_1.
\end{align*} 
The bound $\textrm{gon}(G\boxempty K_3)\leq 3\textrm{gon}(G)\leq 3(g(G)+1)=3e_1-3v_1+6$ implies
\begin{align*}
    \frac{g(G\boxempty K_3)+3}{2}-\textrm{gon}(G\boxempty K_3)\,\geq\, & \frac{3e_1}{2}+2-3e_1+3v_1-6.
    \\=& -\frac{3e_1}{2}-4+3v_1
    \\\,\geq\, & -\frac{3e_1}{2}-2+2v_1,
\end{align*}
where we use the fact that $v_1\geq 2$.  At least one of $\frac{3e_1}{2}+2-2v_1$ and $-\frac{3e_1}{2}-2+2v_1$ is nonnegative, so we have  $\frac{g(G\boxempty K_3)+3}{2}-\textrm{gon}(G\boxempty K_3)\geq 0$. As when $H=K_2$, it follows that $\textrm{gon}(G\boxempty K_3)\leq \floor{\frac{g(G\boxempty K_3)+3}{2}}$.  Note that the exact same arguments suffice to show that $\textrm{gon}(G\boxempty B_{2,1})\leq \floor{\frac{g(G\boxempty B_{2,1})+3}{2}}$, since $B_{2,1}$ also has $3$ vertices, $3$ edges, and gonality equal to $2$.

Finally, let $H=B_2$.  We have $v_2=2$ and $e_2=2$, so $g(G\cart H)=e_1v_2+e_2v_1-v_1v_2+1=2e_1+2v_1-2v_1+1=2e_1+1$, meaning $\frac{g(G\cart H)+3}{2}=e_1+2$.  We have $\gon(G\cart B_2)\leq\min\{2\gon(G),2v_1\}$, so $\gon(G\cart B_2)\leq 2v_1$, and thus
\[ \frac{g(G\boxempty B_2)+3}{2}-\textrm{gon}(G\boxempty B_2)\,\geq\,  e_1+2-2v_1.\]
The bound $\gon(G\cart B_2)\leq 2\gon(G)\leq 2(g(G)+1)=2e_1-2v_1+4$ gives us
\[ \frac{g(G\boxempty B_2)+3}{2}-\textrm{gon}(G\boxempty B_2)\,\geq\,  e_1+2-(2e_1-2v_1+4)=-e_1-2+2v_1.\]
At least one of $e_1+2-2v_1$ and $-e_1-2+2v_1$ is nonnegative.  It follows that $\gon(G\cart B_2)\leq \floor{\frac{g(G\cart B_2)+3}{2}}$.
\end{proof}

\section{Products with gonality equal to $\floor{\frac{g+3}{2}}$}\label{section:equality}

We now determine which nontrivial graph products $G\boxempty H$  satisfy $\textrm{gon}(G\boxempty H)=\lfloor \frac{g(G\boxempty H)+3}{2}\rfloor$.     We start with the following lemma, which determines the gonality of several graph products not yet dealt with.

\begin{lemma}\label{lemma:odds_and_ends}
We have $\gon(B_2\cart B_2)=4$, and $\gon(B_{2,1}\cart K_3)=6$.
\end{lemma}

These graphs are the rightmost pair of graphs in the bottom row of Figure \ref{figure:multigraphs_with_equality}.

\begin{proof}
As usual, Proposition \ref{prop:upperbound} furnishes the desired upper bound on these gonalities.  We will argue lower bounds using Lemma \ref{lemma:burning}.

Suppose for the sake of contradiction that there exists an effective divisor $D$ of degree $3$ and positive rank on $B_2\cart B_2$.  Since $B_2\cart B_2$ has four vertices, we can choose $v$ with no chips on it, and run the burning process from Lemma \ref{lemma:burning}.  Since two edges connect each vertex to each of its neighbor, a vertex could only be safe from a burning neighbor if it had $2$ or more chips.  Thus as the fire spreads around the underlying cycle $C_4$ of the graph, at most one vertex is not burned; but then there are $4$ incident burning edges for that vertex, and at most $3$ chips on the graph, so the whole graph burns.  This contradicts $r(D)>0$, so $\gon(B_2\cart B_2)\geq 4$.

Suppose for the sake of contradiction that there exists an effective divisor $D$ of degree $5$ and positive rank on $B_{2,1}\cart K_3$.  Label the vertices of $B_{2,1}$ as $u$, $v$, and $w$, where $\val(u)=2$, $\val(v)=3$, and $\val(w)=1$; and let $C$ be the cycle connecting $u$ and $v$. Refer to the three copies of $B_{2,1}$ as $G_1$, $G_2$, and $G_3$; and to their vertices as $u_i$, $v_i$, and $w_i$ and cycle as $C_i$ in accordance with our labelling on $G$.  We will refer to the three copies of $K_3$ as $u$-$K_3$, $v$-$K_3$, and $w$-$K_3$, depending on which vertices they are composed of.  This notation is illustrated in Figure \ref{figure:labels}.

 \begin{figure}[hbt]
   		 \centering
\includegraphics[scale=1]{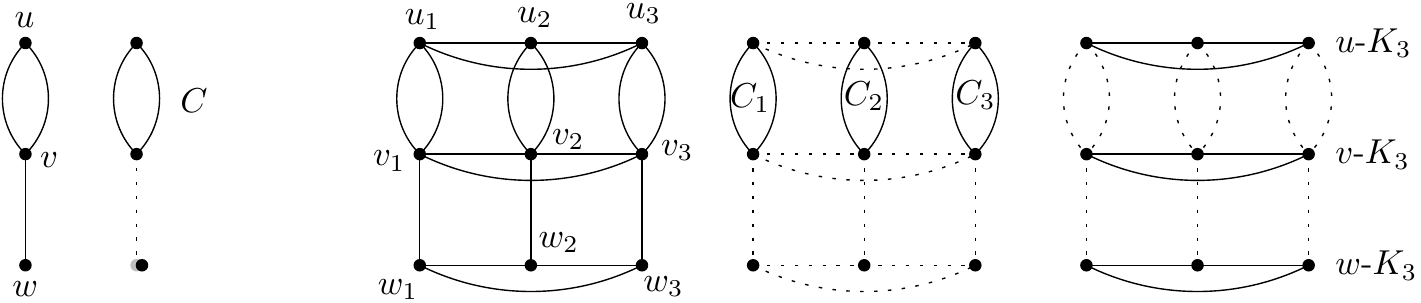}
	\caption{Our labels on $B_{2,1}$ and $B_{2,1}\cart K_3$} 
	\label{figure:labels}
\end{figure}

First we rule out several possible cases for $D$, namely if $D$ is of the form $(u_i)+(v_i)+(u_j)+(v_j)+(w_k)$ where $i$, $j$, and $k$ are all distinct;  and if $D$ is of the form $2(u_i)+3(v_i)$.  It turns out a divisor of the first form is equivalent to a divisor of the second form, as illustrated in Figure \ref{figure:b21k3_divisors}.  Neither divisor has positive rank:  if $D=(u_i)+(v_i)+(u_j)+(v_j)+(w_k)$, we can run the burning process from $w_i$ or $w_j$, and the whole graph burns; and if $D=2(u_i)+3(v_i)$, then it is equivalent to a divisor of the previous form, which does not have positive rank.  Thus we may assume that $D$ does not have either of those forms.

 \begin{figure}[hbt]
   		 \centering
\includegraphics[scale=0.9]{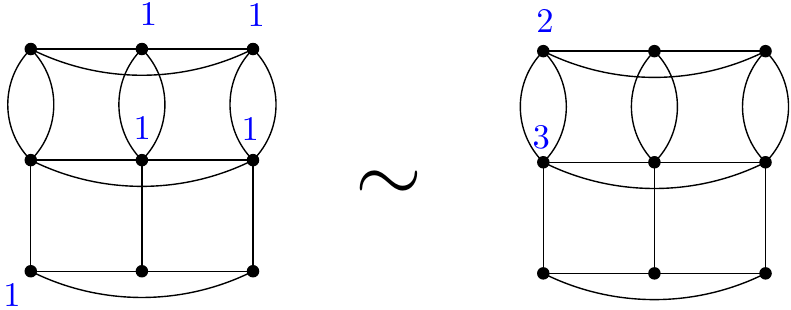}
	\caption{Two types of divisors, equivalent to one another, that we rule out separately} 
	\label{figure:b21k3_divisors}
\end{figure}

Since $g(B_{2,1}\cart K_3)=10>5$, we may assume by Lemma \ref{lemma:spencer} that $D=D_{\mathcal{O}}$ for some sourceless partial orientation $\mathcal{O}$, implying that $D$ places at most $\val(x)-1$ chips on a vertex $x$. If $D$ places a chip on each of $w_1$, $w_2$, and $w_3$, then fire these three vertices to move one chip each to $v_1$, $v_2$, and $v_3$; and if $D$ places at least two chips on two of $w_1,w_2,$ and $w_3$ (but no chip on the third), fire those two vertices to move chips away. Replacing our $D$ with this new divisor, we still have that it places at $\val(x)-1$ chips on a vertex $x$, since there were at most $1$ or $2$ chips off of $w$-$K_3$, depending on which $w$-configuration we were handling.

Since there are three copies of $C$, at least one of them, say $C_1$, has at most one chip.  Choose a vertex on $C_1$ without a chip and run the burning process. Certainly all of $C_1$ burns.  Let $b$ be the number of vertices of $C_2\cup C_3$ that burn.  We now deal with several cases.

\begin{itemize}
    \item Suppose $b=0$.  Then each of the four vertices of $C_2\cup C_3$ have at least one chip.  Since there are only five chips at least one of $v_2$ or $v_3$ must only have one chip, meaning that we cannot have the $w$-$K_3$ burn.  With at most one chip not on $C_2\cup C_3$, the only way to prevent the $w$-$K_3$ from burning is to place $1$ chip on $w_1$.  But then $D$ is of the form $(u_i)+(v_i)+(u_j)+(v_j)+(w_k)$ where $i$, $j$, and $k$ are all distinct, a contradiction.
    \item  Suppose $b=1$; say it is $u_2$ that burns.  Counting up burning edges and noting that no other vertex of $C_2\cup C_3$ burns, we know $v_2$ has three chips, $u_3$ has two chips, and $v_3$ has one chip.  This exceeds our five chips, a contradiction.  The same argument works for any other vertex of $C_2\cup C_3$.
    \item  Suppose $b=2$.  If $u_i$ and $v_j$ burn for $i\neq j$, then $v_i$ and $u_j$ each need $3$ chips, which is impossible.  If $u_i$ and $u_j$ burn, then $v_i$ and $v_j$ each need $3$ chips, which is impossible; a similar contradiction occurs if $v_i$ and $v_j$ burn.  If $u_i$ and $v_i$ burn, then each of $u_i$ and $v_j$ needs at least $2$ chips.  In order to prevent $v_i$ from burning, it must have a third chip:  there is no way to use $1$ chip to prevent $w_i$ from burning.  But then $D$ is of the form $2(u_i)+3(v_i)$, a contradiction.
    
    \item  Suppose $b=3$.  If $u_i$ does not burn, then since all vertices incident to it are burning, it must have $4$ chips, a contradiction since it can have at most $\val(u_i)-1=3$ chips.  If $v_i$ does not burn, then it must have $4$ chips, which is not yet a contradiction since $\val(v_i)-1=4$.  However, since $C_j$ and $C_k$ are burning for $j\neq k$, $1$ chip is not enough to prevent $w$-$K_3$ from burning.  This means $w_i$ is burning, so $v_i$ needs $5$ chips, a contradiction.
\end{itemize}
Thus we have that $b=4$, so all of $C_1\cup C_2\cup C_3$ burns.  We know at least one $w_i$ has no chips, so that vertex burns; and the other $w_j$ and $w_k$ now have $2$ incoming burning edges each.  By our construction of $D$ we know they can't both have $2$ chips, so one burns; and then the other burns as well, since it can't have more than $2$ chips.  We conclude the whole graph burns, contradicting $r(D)>0$.  We conclude that $\gon(B_{1,2}\cart K_3)\geq 6$.
\end{proof}

We are now ready to state the main theorem of this section. 

\begin{theorem}\label{theorem:equality}
Let $G$ and $H$ be graphs with at least two vertices each.  Then $\textrm{gon}(G\boxempty H)=\floor{\frac{g(G\boxempty H)+3}{2}}$ if and only if $G\boxempty H$ is $K_2\boxempty K_2$, $K_2\boxempty P_3$, $P_3\boxempty P_3$, $K_3\boxempty K_3$,  $B_{1,2}\cart K_3$, $B_2\cart B_2$, or $O\boxempty K_2$ where $O$ is a genus $1$ graph with $3\leq |V(O)|\leq 5$.
\end{theorem}

Before we prove this theorem, we enumerate the possible graphs $O$ of genus $1$ with between $3$ and $5$ vertices.  Such an $O$ consists of a single cycle with $c\geq 2$ vertices, with up to $5-c$ other vertices connected to this cycle without forming a new cycle; note that $c=2$ if and only if $G$ is not simple.  There end up being $8$ such simple graphs, all illustrated in the top row of Figure \ref{figure:genus_1_345}:  the cycles $C_3$, $C_4$, and $C_5$; the so-called \emph{tadpole graphs} $T_{3,1}$, $T_{3,2}$, and $T_{4,1}$; the so-called \emph{bull graph} $B$; and the so-called \emph{cricket graph} $K$.  Note that the cycle graph $C_3$ is equal to the complete graph $K_3$.  There are $9$ non-simple graphs, pictured on the second row of Figure \ref{figure:genus_1_345}.

\begin{figure}[hbt]
   		 \centering
\includegraphics[scale=0.8]{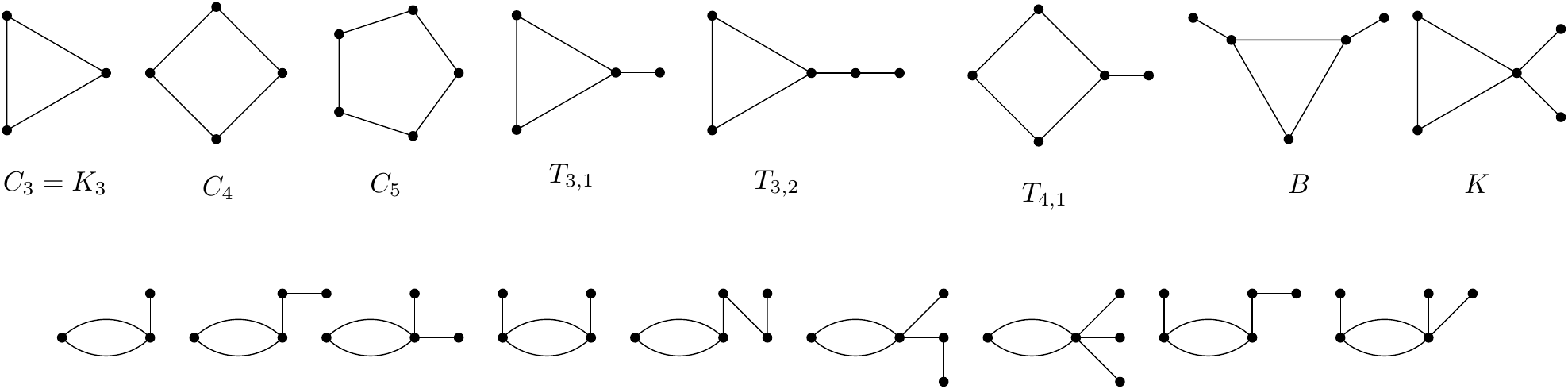}
	\caption{The graphs of genus $1$ with between $3$ and  $5$ vertices, simple and non-simple}
	\label{figure:genus_1_345}
\end{figure}


We illustrate the content of Theorem \ref{theorem:equality} with Table \ref{table:products} and with Figure \ref{figure:multigraphs_with_equality}, which respectively show all simple and non-simple nontrivial graph products $G\boxempty H$ satisfying $\textrm{gon}(G\boxempty H)=\floor{\frac{g(G\boxempty H)+3}{2}}$.  We include references in the table to results from this paper that imply the claimed gonality, although many of these gonalities were previously known.  In particular, the gonality of the $m\times n$ grid was proved to be $\min\{m,n\}$ in \cite{db}; and the gonality of $C_m\boxempty P_n$ was proved to be $\min\{m,2n\}$ for $m\neq 2n$ in \cite{treewidth}, giving the gonality of the $2\times 3$ rook's graph and the $5$-prism.

\begin{table}
\begin{center}
\begin{tabular}{ |c|c|c|c|c| } 
  \hline
 Product & Name or description & Graph & Gonality & Source for gonality\\ 
 \hline
 $K_2\boxempty K_2$ & $2\times 2$ grid graph & \includegraphics{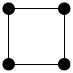} & 2& Proposition \ref{prop:tree_tree} \\ 
 \hline
$K_2\boxempty P_3$ & $2\times 3$ grid graph& \includegraphics{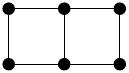} & 2& Proposition \ref{prop:tree_tree}\\ 
 \hline
  $P_3\boxempty P_3$ & $3\times 3$ grid graph & \includegraphics{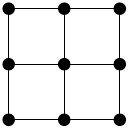} & 3& Proposition \ref{prop:tree_tree} \\ 
 \hline
 $K_3\boxempty K_3$ & $3\times 3$ rook's graph & \includegraphics{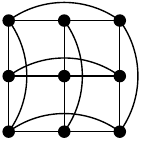} & 6& Theorem \ref{theorem:rooks}\\
 \hline
 $K_3\boxempty K_2$ & $2\times 3$ rook's graph & \includegraphics{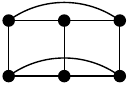} & 3& Proposition \ref{prop:tree_complete}, Theorem \ref{theorem:rooks}\\ 
 \hline
  $C_4\boxempty K_2$ & $4$-prism, or $3$-cube & \includegraphics[scale=0.5]{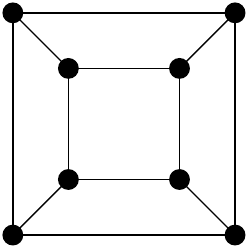} & 4& Theorem \ref{theorem:genus_1_345}\\ 
 \hline
  $C_5\boxempty K_2$ & $5$-prism & \includegraphics[scale=0.5]{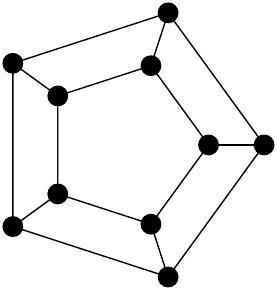} & 4& Theorem \ref{theorem:genus_1_345}\\ 
 \hline
   $ T_{3,1}\boxempty K_2$ & $3$-prism with flap & \includegraphics[scale=0.5]{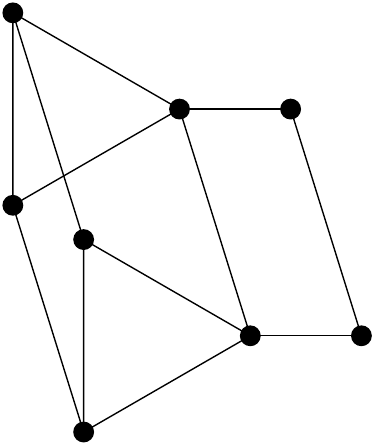} & 4& Theorem \ref{theorem:genus_1_345}\\ 
 \hline
   $ T_{3,2} \boxempty K_2$ & $3$-prism with long flap  & \includegraphics[scale=0.5]{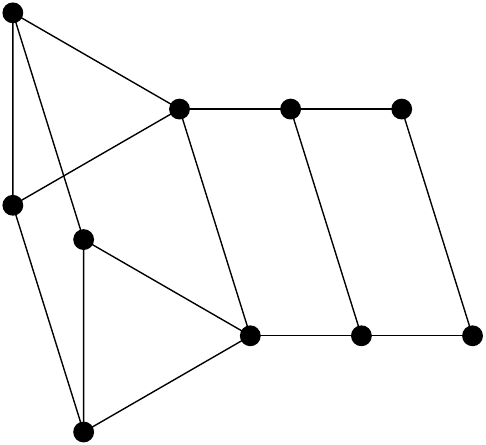} & 4& Theorem \ref{theorem:genus_1_345}\\ 
 \hline
   $ T_{4,1}\boxempty K_2$ & $4$-prism with flap & \includegraphics[scale=0.5]{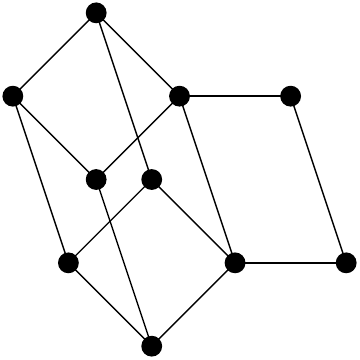} & 4& Theorem \ref{theorem:genus_1_345}\\ 
 \hline
   $B \boxempty K_2$ & $3$-prism with two flaps & \includegraphics[scale=0.5]{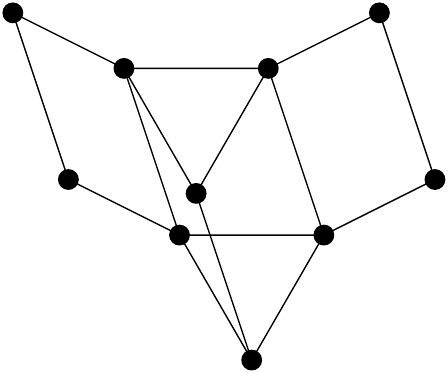} & 4& Theorem \ref{theorem:genus_1_345}\\ 
 \hline
    $K\boxempty K_2$ &$3$-prism with two adjacent flaps & \includegraphics[scale=0.5]{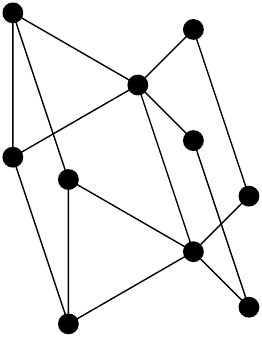} & 4& Theorem \ref{theorem:genus_1_345}\\ 
 \hline
\end{tabular}
\end{center}
\caption{All nontrivial simple graph products $G\boxempty H$ with gonality equal to $\floor{\frac{g(G\boxempty H)+3}{2}}$}
\label{table:products}
\end{table}

 \begin{figure}[hbt]
   		 \centering
\includegraphics[scale=0.9]{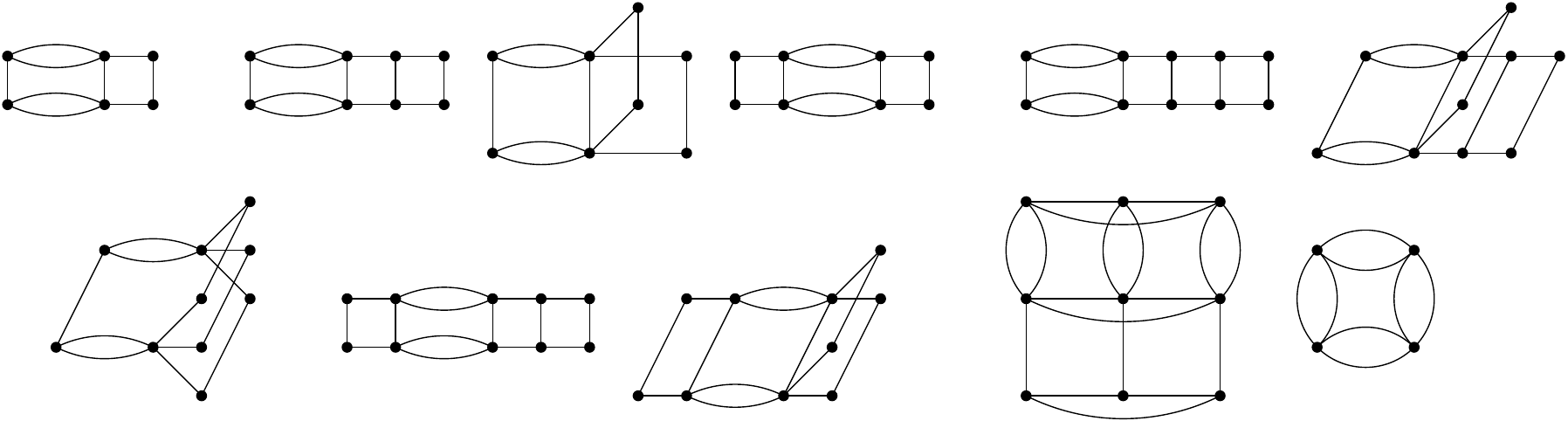}
	\caption{All nontrivial products $G\cart H$ with $\gon(G\cart H)=\floor{\frac{g(G\cart H)+3}{2}}$, where $G$ and $H$ are not both simple} 
	\label{figure:multigraphs_with_equality}
\end{figure}

\begin{proof}[Proof of Theorem \ref{theorem:equality}]
As before, we may assume without loss of generality that $e_2v_1\leq e_1v_2$.  By the proof of Theorem \ref{theorem:main}, if $v_2\geq 4$ or $g(H)\geq 2$ there is a gap between $\textrm{gon}(G\boxempty H)$ and $\lfloor \frac{g+3}{2}\rfloor$.  Thus to have equality we must have $v_2<4$ and $g(H)\leq 1$, so $H$ must be one of the graphs $K_2$,  $P_3$, $K_3$, $B_{2,1}$, and $B_2$.  For some of these cases, we will consider $\Delta(G\boxempty H):=\frac{g(G\boxempty H)+3}{2}-\textrm{gon}(G\boxempty H)$.  In any case where $\Delta(G\boxempty H)\geq 1$, we will have $\textrm{gon}(G\boxempty H)<\lfloor \frac{g(G\boxempty H)+3}{2}\rfloor$.

Let $H=K_2$.  We will deal with three cases sorted by the genus $g$ of $G$:  $g=0$, $g=1$, and $g\geq 2$.
\begin{itemize}
    \item 
If $g=0$, then $G$ is a tree, so $\textrm{gon}(G\boxempty K_2)=\min\{2,v_1\}=2$ by Proposition \ref{prop:tree_tree}.  We then have $v_1=e_1-1$, so $\floor{\frac{g(G\boxempty K_2)+3}{2}}=\floor{\frac{e_1v_2+e_2v_1-v_1v_2+4}{2}}=\floor{\frac{2(v_1-1)+v_1-2v_1+4}{2}}=\floor{\frac{v_1+2}{2}}$.  This is equal to $2$ if and only if $v_1=2$ or $v_1=3$, so we must have $G=K_2$ or $G=P_3$. This gives us the graphs $K_2\boxempty K_2$ and $K_2\boxempty P_3$.


\item  Next assume that $G$ has genus $1$.  We know that $\gon(G\cart K_2)=\min\{v_1,4\}$ by  Theorem \ref{theorem:genus_1_345}.  Note that $\floor{\frac{g(G\boxempty K_2)+3}{2}}=\floor{e_1-\frac{v_1}{2}}+2=\floor{\frac{v_1}{2}}+2$.  If $v_1=3$, then this equals $3$; and if $v_1=4$ or $v_1=5$, then this equals $4$.  In both cases, we do have that $\floor{\frac{g(G\boxempty K_2)+3}=\min\{v_1,4\}$.  If $v_1=2$, then  $\floor{\frac{v_1}{2}}+2=3>2=\gon(G\cart K_2)}$; and if $v_1\geq 6$, then $\floor{\frac{g(G\boxempty K_2)+3}{2}}>4=\gon(G\cart K_2)$.  Thus if $G$ has genus $1$, we have  $\gon(G\cart K_2)=\floor{\frac{g(G\boxempty K_2)+3}{2}}$ if and only if $3\leq v_1\leq 5$.  This gives us $17$ products of the form $O\cart K_2$, where $O$ is any of the graphs in Figure \ref{figure:genus_1_345}.

\item  Assume now that $G$ has genus $g\geq 2$. We will show that in this case $\Delta(G\boxempty K_2)\geq 1$.  We know from the proof of Theorem \ref{theorem:main} that $\Delta(G\boxempty K_2)\geq e_1-\frac{3v_1}{2}+2$.  Since $g\geq 2$, we know by Lemma \ref{lemma:gon_g} that $\textrm{gon}(G)\leq g$, so $\textrm{gon}(G\boxempty K_2)\leq 2\textrm{gon}(G)\leq 2g=2e_1-2v_1+2$. It follows that
\begin{align*}
    \Delta(G\boxempty K_2)\geq& e_1-\frac{v_1}{2}+2 -(2e_1-2v_1+2)
    \\=& -e_1+\frac{3v_1}{2}.
\end{align*}
If $e_1-\frac{3v_1}{2}+2\geq 1$, then we have our desired lower bound on $\Delta(G\boxempty K_2)$.  Otherwise, we have $e_1-\frac{3v_1}{2}+2\leq \frac{1}{2}$, implying $-e_1+\frac{3v_1}{2}-2\geq -\frac{1}{2}$, and so $-e_1+\frac{3v_1}{2}\geq \frac{3}{2}$.  Since $\Delta(G\boxempty K_2)\geq -e_1+\frac{3v_1}{2}$, we still have $\Delta(G\boxempty K_2)\geq 1$.  We conclude that $\Delta(G\boxempty K_2)\geq 1$, meaning that $\textrm{gon}(G\boxempty K_2)<\floor{\frac{g(G\boxempty K_2)}{2}}$.
\end{itemize}

Now let $H=P_3$. A careful reading of the  proof of Theorem \ref{theorem:main} shows that if $G$ is not a tree, then $\textrm{gon}(G\boxempty P_3)< \lfloor \frac{g(G\boxempty P_3)+3}{2}\rfloor$; and that if $G$ is a tree, then $\floor{\frac{g(G\boxempty P_3)+3}{2}}=v_1$ and $\gon(G\boxempty P_3)=\min\{3,v_1\}$, so
$\textrm{gon}(G\boxempty P_3)= \lfloor \frac{g(G\boxempty P_3)+3}{2}\rfloor$ if and only if $v_1=\min\{v_1,3\}$.  It follows that $v_1$ must be either $2$ or $3$, and thus $G$ must be either $K_2$ or $P_3$.

Now let $H=K_3$. From the proof of Theorem \ref{theorem:main}, we have
\[\Delta(G\boxempty K_3)\geq \max\left\{\frac{3e_1}{2}+2-2v_1,-\frac{3e_1}{2}-4+3v_1\right\}.\]
If we have $\frac{3e_1}{2}+2-2v_1>\frac{1}{2}$, then $\Delta(G\boxempty K_3)\geq \frac{3e_1}{2}+2-2v_1\geq 1$ and we're done.  If not, then $\frac{3e_1}{2}+2-v_1\leq \frac{1}{2}$.  We then have $-\frac{3e_1}{2}-2+2v_1\geq -\frac{1}{2}$, and from there that $-\frac{3e_1}{2}-4+3v_1\geq -\frac{5}{2}+v_1$.  We thus have $\Delta(G\boxempty K_3)\geq -\frac{5}{2}+v_1$.  If $v_1\geq \frac{7}{2}$, then we have $\Delta(G\boxempty K_3)\geq 1$, so the only way we could have $\textrm{gon}(G\boxempty K_3)=\lfloor \frac{g(G\boxempty K_3)+3}{2}\rfloor$ is if $v_1<\frac{7}{2}$, which implies that $v_1=2$ or $v_1=3$.
 If $v_1=2$, the bound  $\Delta(G\cart K_3)\geq \frac{3e_1}{2}+2-2v_1$ becomes $\Delta(G\cart K_3)\geq \frac{3e_1}{2}-2$; and if $v_1=3$, it becomes $\Delta(G\cart K_3)\geq \frac{3e_1}{2}-4$.  If $v_1=2$ and $e_1\geq 2$, then $\Delta(G\cart K_3)\geq\frac{3\cdot 2}{2}-2=1$; and if $v_1=3$ and $e_1\geq 4$, then  $\frac{3e_1}{2}-4\geq \frac{3\cdot 4}{2}-4=2$. In both these cases we have $\Delta(G\cart K_3)\geq1$, implying a gap. Thus the only possible cases for equality are when $v_1=2$ and $e_1=1$; and when $v_1=3$ and $e_1\leq 3$.  The only graphs satisfying these properties are $K_2$, $P_3$, $K_3$, and $B_{2,1}$.  We can rule $P_3$ out since $K_3$ is not a tree. 
By Proposition \ref{prop:tree_complete}, Theorem \ref{theorem:rooks}, and Lemma \ref{lemma:odds_and_ends}, we have $\textrm{gon}(K_2\boxempty K_3)=3$, $\textrm{gon}(K_3\boxempty K_3)=6$, and $\textrm{gon}(B_{2,1}\boxempty K_3)=6$.  On the other hand, we have $\floor{\frac{g(K_2\boxempty K_3)+1}{2}}=3$, $\floor{\frac{g(K_3\boxempty K_3)+1}{2}}=6$, and $\floor{\frac{g(P_3\boxempty K_3)+1}{2}}=6$.  Thus we do have equality for the three products $K_2\boxempty K_3$, $K_3\boxempty K_3$, and $B_{2,1}\boxempty K_3$ (the first of which we already knew from our analysis of $K_2$).

Since $B_{2,1}$ and $K_3$ both have $3$ edges, $3$ vertices, genus $1$, and gonality $2$, an identical argument shows that if $H=B_{2,1}$, then we need $G\in\{K_2,K_3,B_{2,1}\}$ in order to have equality.  We do have $\gon(G\cart B_{2,1})=\floor{\frac{g(G\cart B_{2,1})+3}{2}}$ for $G=K_2$ and $G=K_3$, as already determined earlier in this proof.  However, $\gon(B_{2,1}\cart B_{2,1})\leq 5$ by Figure \ref{figure:counterexample}, and $\floor{\frac{g(B_{2,1}\cart B_{2,1})+3}{2}}=\floor{\frac{10+3}{2}}=6$, so we do not have equality in this case. 

 Finally, if $H=B_2$, recall that we have the lower bounds $\Delta(G\cart B_2)\geq e_1-2v_1+2$ and $\Delta(G\cart B_2)\geq-e_1+2v_1-2$.  One of these lower bounds implies $\Delta(G\cart B_2)\geq 1$ unless $e_1-2v_1+2=-e_1+2v_1-2=0$.  Suppose we are in this latter case, which implies that $g(G)=e_1-v_1+1=v_1-1$.  We deal with two cases: when $G$ has genus at most $1$, and when $g(G)\geq 2$. 

\begin{itemize}
\item        If $g(G)=0$, then $v_1-1=0$ and $v_1=1$, a contradiction.  If $g(G)=1$, then $v_1=2$.  The only graph of genus $1$ with $2$ vertices is $B_2$; we do indeed have that $\gon(B_2\cart B_2)=4=\floor{\frac{g(B_2\cart B_2)+3}{2}}$ by Lemma \ref{lemma:odds_and_ends}. This graph is the rightmost graph on the bottom row in Figure \ref{figure:multigraphs_with_equality}.

    \item      If $g(G)\geq 2$, we can improve the bound $\Delta(G\cart B_2)\geq -e_1+2v_1-2$ to $\Delta(G\cart B_2)\geq -e_1+{2}v_1$, since we have $\gon(G)\leq 2g(G)$.  Since $-e_1+2v_1-2=0$, we have $-e_1+2v_1=2$.  This lower bound of $2$ on  $\Delta(G\cart B_2)$ implies we cannot have equality.

\end{itemize}     
\end{proof}
\medskip

\noindent \textbf{Acknowledgements.}  The authors would like to thank Franny Dean, David Jensen, Nathan Pflueger, Teresa Yu, and Julie Yuan for many helpful conversations about product graphs and graph gonality.  The authors are grateful for support they received from NSF Grants DMS1659037 and DMS1347804, and from the Williams College SMALL REU program.

\bibliographystyle{alpha}
\bibliography{bibliography}

\newcommand{\etalchar}[1]{$^{#1}$}
\begin{thebibliography}{ADM{\etalchar{+}}b}

\bibitem[ADM{\etalchar{+}}a]{gon3}
Ivan Aidun, Frances Dean, Ralph Morrison, Teresa Yu, and Julie Yuan.
\newblock Graphs of gonality three.
\newblock To appear in \emph{Algebraic Combinatorics}.

\bibitem[ADM{\etalchar{+}}b]{treewidth}
Ivan Aidun, Frances Dean, Ralph Morrison, Teresa Yu, and Julie Yuan.
\newblock Treewidth and gonality of glued grid graphs.
\newblock To appear in \emph{Discrete Applied Mathematics}.

\bibitem[AR18]{ar}
Stanislav Atanasov and Dhruv Ranganathan.
\newblock A note on {B}rill-{N}oether existence for graphs of low genus.
\newblock {\em Michigan Math. J.}, 67(1):175--198, 2018.

\bibitem[Bac17]{spencer}
Spencer Backman.
\newblock Riemann-{R}och theory for graph orientations.
\newblock {\em Adv. Math.}, 309:655--691, 2017.

\bibitem[Bak08]{baker}
Matthew Baker.
\newblock Specialization of linear systems from curves to graphs.
\newblock {\em Algebra Number Theory}, 2(6):613--653, 2008.
\newblock With an appendix by Brian Conrad.

\bibitem[BN07]{bn}
Matthew Baker and Serguei Norine.
\newblock Riemann-{R}och and {A}bel-{J}acobi theory on a finite graph.
\newblock {\em Adv. Math.}, 215(2):766--788, 2007.

\bibitem[BN09]{bn2}
Matthew Baker and Serguei Norine.
\newblock Harmonic morphisms and hyperelliptic graphs.
\newblock {\em Int. Math. Res. Not. IMRN}, (15):2914--2955, 2009.

\bibitem[Cap14]{cap}
Lucia Caporaso.
\newblock Gonality of algebraic curves and graphs.
\newblock In {\em Algebraic and complex geometry}, volume~71 of {\em Springer
  Proc. Math. Stat.}, pages 77--108. Springer, Cham, 2014.

\bibitem[CD18]{cd}
Filip Cools and Jan Draisma.
\newblock On metric graphs with prescribed gonality.
\newblock {\em J. Combin. Theory Ser. A}, 156:1--21, 2018.

\bibitem[CMP19]{pathwidth}
Nancy~Ellen Clarke, Margaret-Ellen Messinger, and Grace Power.
\newblock Bounding the search number of graph products.
\newblock {\em Kyungpook Math. J.}, 59(1):175--190, 2019.

\bibitem[Dha90]{dhar}
Deepak Dhar.
\newblock Self-organized critical state of sandpile automaton models.
\newblock {\em Phys. Rev. Lett.}, 64(14):1613--1616, 1990.

\bibitem[{Gij}15]{gij}
D.~{Gijswijt}.
\newblock {Computing divisorial gonality is hard}.
\newblock {\em ArXiv e-prints}, April 2015.

\bibitem[GK08]{gk}
Andreas Gathmann and Michael Kerber.
\newblock A {R}iemann-{R}och theorem in tropical geometry.
\newblock {\em Math. Z.}, 259(1):217--230, 2008.

\bibitem[Luc07]{lucena}
Brian Lucena.
\newblock Achievable sets, brambles, and sparse treewidth obstructions.
\newblock {\em Discrete Appl. Math.}, 155(8):1055--1065, 2007.

\bibitem[MZ08]{mz}
Grigory Mikhalkin and Ilia Zharkov.
\newblock Tropical curves, their {J}acobians and theta functions.
\newblock In {\em Curves and abelian varieties}, volume 465 of {\em Contemp.
  Math.}, pages 203--230. Amer. Math. Soc., Providence, RI, 2008.

\bibitem[ST93]{st}
Paul~D. Seymour and Robin Thomas.
\newblock Graph searching and a min-max theorem for tree-width.
\newblock {\em Journal of Combinatorial Theory}, 58:22--33, 1993.

\bibitem[vDdB12]{db}
Josse van Dobben~de Bruyn.
\newblock {\em Reduced divisors and gonality in finite graphs}.
\newblock Bachelor's thesis, Mathematisch Instituut, Universiteit Leiden, 2012.

\bibitem[vG14]{vddbg}
J.~{van Dobben de Bruyn} and D.~{Gijswijt}.
\newblock {Treewidth is a lower bound on graph gonality}.
\newblock {\em ArXiv e-prints}, July 2014.

\end{thebibliography}

\end{document}